\documentclass[11pt]{amsart}
\usepackage[utf8]{inputenc}
\usepackage{tabularx}
\usepackage{supertabular}
\usepackage{mathrsfs}
\usepackage{mathtools}
\usepackage{enumitem}
\usepackage{tikz-cd}
\usepackage[all]{xy}

\usepackage{hyperref, amsmath, amssymb, amsthm, thmtools, fullpage, environ,enumitem}
\usepackage[capitalise, nameinlink]{cleveref}

\usepackage{todonotes}

\NewEnviron{eqsplit}{
\begin{equation*}
\begin{split}
  \BODY
\end{split}
\end{equation*}
}

\def\R{{\mathbb{R}}}
\def\N{{\mathbb{N}}}
\def\C{{\mathbb{C}}}
\def\G{{\overline{G}}}

\def\Z{{\mathbb{Z}}}
\def\I{{\mathcal{I}}}
\def\J{{\mathcal{J}}}

\newcommand{\<}{\vartriangleleft}
\newcommand{\lf}{\leq_{fin}}
\newcommand{\nf}{\vartriangleleft_{fin}}
\DeclareMathOperator{\Ann}{\text{Ann}}
\DeclareMathOperator{\Aut}{\text{Aut}}
\DeclareMathOperator{\dr}{\dim_{\text{Rok}}}
\newcommand{\ssubset}{\subset\mathrel{\mkern-9mu}\subset}

\newtheorem{theorem}{Theorem}[section]
\newtheorem{prop}[theorem]{Proposition}
\newtheorem{cor}[theorem]{Corollary}
\newtheorem{lemma}[theorem]{Lemma}

\theoremstyle{definition}
\newtheorem{definition}[theorem]{Definition}
\newtheorem{rem}[theorem]{Remark}

\hypersetup{
colorlinks,
linkcolor=blue,          
filecolor=magenta,      
urlcolor=cyan           
}

\title{Rokhlin Dimension: Permanence Properties and Ideal Separation}
\author{Sureshkumar M, Prahlad Vaidyanathan}
\begin{document}

\begin{abstract}
We study the Rokhlin dimension for actions of residually finite groups on C*-algebras. We give a definition equivalent to the original one due to Szab\'{o}, Wu and Zacharias. We then prove a number of permanence properties and discuss actions on $C_0(X)$-algebras and commutative C*-algebras. Finally, we use a theorem of Sierakowski to show that, for an action with finite Rokhlin dimension, every ideal in the associated reduced crossed product C*-algebra arises from an invariant ideal of the underlying algebra.
\end{abstract}

\maketitle

The study of group actions on C*-algebras is a deep and integral part of the general theory of operator algebras. It allows us to understand the underlying algebra by studying its symmetries, and also provides us with a new C*-algebra, the crossed product.  This crossed product is a fascinating object that (broadly speaking) contains the original algebra and the group, and implements the group action via conjugation by unitaries. The difficulty then is to understand the structure of this crossed product C*-algebra, and a key question in this context is to determine when certain `regularity' properties pass from the underlying algebra to the crossed product. These regularity properties include many properties that are useful from the point of view of the classification programme such as finiteness of nuclear dimension, simplicity, $\mathcal{Z}$-stability, etc. \\

The Rokhlin property (studied by Kishimoto \cite{kishimoto_shifts}, Izumi \cite{izumi} and others) was a first attempt in this direction, but it has the crucial problem in that it requires the underlying algebra to have sufficiently many projections. The notion of Rokhlin dimension, introduced by Hirshberg, Winter and Zacharias \cite{hirshberg}, seeks to avoid this problem by replacing projections by positive contractions. In the past decade or so, this idea has proved to be very fruitful. In \cite{hirshberg}, the authors studied such actions of finite groups and of the group of integers. This was generalized to actions of compact groups by Gardella \cite{gardella_compact} and to actions of residually finite groups by Szab\'{o}, Wu and Zacharias \cite{szabo_wu_zacharias}. In each case, it was proved that actions with finite Rokhlin dimension allow us to deduce a number of regularity properties of the crossed product C*-algebra from those of the underlying algebra (particularly in the compact case, and under the `commuting towers' condition \cite{gardella_hirshberg_santiago}). \\

The goal of this paper is to study actions of countable, discrete, residually finite groups with finite Rokhlin dimension, building on the work done in \cite{szabo_wu_zacharias} and \cite{gardella_compact}. We begin in \cref{sec_preliminaries}, where we discuss the profinite completion of a residually finite group. This artifact allows us to prove results for residually finite groups along the same lines as those of compact groups. We also describe the topological join of two compact spaces, a notion that gets used in the discussion on Rokhlin dimension with commuting towers. In \cref{sec_definition}, we recall the definition of Rokhlin dimension from \cite{szabo_wu_zacharias}, and give a number of equivalent formulations. In \cref{sec_permanence}, we show that finiteness of Rokhlin dimension is preserved under many standard constructions such as passage to hereditary subalgebras, quotients, inductive limits and extensions. Moreover, we show that finiteness of Rokhlin dimension passes to the action of a subgroup provided the subgroup is a virtual retract. \\

In \cref{sec_cx_algebras}, we consider actions on commutative C*-algebras. After proving an estimate for the Rokhlin dimension of an action on a $C_0(X)$-algebra, we show that an action on a commutative C*-algebra has finite Rokhlin dimension if the maximal ideal space has finite covering dimension and the induced action on it is both free and proper. In \cref{sec_outerness}, we show that actions with finite Rokhlin dimension are outer, and are properly outer if the group is amenable and has the (VRC) property (see \autoref{defn_vrc_lr}). Along the way, we use a theorem of Sierakowski \cite{sierakowski} to show that if an action is exact and has finite Rokhlin dimension, then the ideals in the associated crossed product C*-algebra must arise from invariant ideals of the underlying C*-algebra.

\section{Preliminaries}\label{sec_preliminaries}

\subsection{The Profinite Completion}

Recall that a discrete group $G$ is said to be residually finite if for each non-identity element $g\in G$, there is a subgroup $H$ of $G$ of finite index such that $g\notin H$. Given such a group $G$, let $\I_G$ denote the set of all normal subgroups of $G$ of finite index, partially ordered by reverse inclusion. In other words, $H\leq K$ if and only if $K\subset H$. Whenever $H\leq K$, there is a homomorphism $\varphi_{K,H} : G/K \to G/H$ given by $gK\mapsto gH$, and this makes the collection $\{G/H, \varphi_{K,H}\}_{\I_G}$ an inverse system of groups. The inverse limit of this system is called the profinite completion of $G$ and is denoted by $\G$. By definition,
\[
\G = \left\lbrace (g_HH)_{H\in \I_G} : g_KH = g_HH \text{ for all } H,K\in \I_G \text{ with } K\subset H\right\rbrace
\]
Note that $\G$ is a group under componentwise multiplication and is a topological space as a subspace of $\prod_{H \in \I_G} G/H$. A profinite group is, by definition, the inverse limit of a surjective inverse system of finite groups, and $\G$ is therefore a profinite group. \\

For each $H\in \I_G$, there is a natural action $\beta^H : G\curvearrowright G/H$ given by $\beta^H_t(gH) := tgH$. The maps $\varphi_{K,H}$ respect these actions, so we have an induced left-translation action $\beta : G\curvearrowright \G$ given by
\[
\beta_t((g_HH)_{H\in \I_G}) := (tg_HH)_{H\in \I_G}.
\]
The most important properties of $\G$ (from our perspective) are listed below, and the interested reader will find proofs of all these facts in \cite{ribes}.

\begin{prop}\label{prop_profinite_completion}
Let $G$ be a discrete, residually finite group and let $\G$ denote its profinite completion.
\begin{enumerate}
\item For each $K\in \I_G$, there is a surjective group homomorphism $\pi_K : \G \to G/K$ given by $(g_HH)_{H\in \I_G} \mapsto g_KK$.
\item If $\{H_1, H_2, \ldots, H_k\}$ is a finite collection of subgroups in $\I_G$ and $\{x_1, x_2,\ldots, x_k\} \subset G$, then the set
\[
\bigcap_{i=1}^k \pi_{H_i}^{-1}(\{x_iH_i\})
\]
is an open set in $\G$, and sets of this type form a basis for the topology on $\G$.
\item $\G$ is compact, Hausdorff and totally disconnected.
\item The map $\iota : G\to \G$ given by $g \mapsto (gH)_{H\in \I_G}$ is an injective group homomorphism and $\iota(G)$ is dense in $\G$.
\item If $H$ is a profinite group and $\varphi : G\to H$ is a group homomorphism, then there is a unique continuous group homomorphism $\overline{\varphi} : \G\to H$ such that $\overline{\varphi}\circ \iota = \varphi$. 
\item The action $\beta:G\curvearrowright \G$ defined above is both free and minimal (in the sense that $\G$ has no non-trivial closed $\beta$-invariant subsets).
\end{enumerate}
\end{prop}

Now consider the commutative C*-algebra $C(\G)$. Since $\G = \varprojlim (G/H, \varphi_{K,H})$, it follows that $C(\G)$ is the inductive limit of the system $\{C(G/H), \varphi_{K,H}^{\ast}\}$. For each $H\in \I_G$, let $\sigma^H : G\to \Aut(C(G/H))$ be the action induced by $\beta^H$ (in other words, $\sigma^H_t(f)(gH) := f(t^{-1}gH)$). Similarly, let $\sigma : G\to \Aut(C(\G))$ be the action induced by $\beta$. Then the maps $\pi_K^{\ast} : C(G/K)\to C(\G)$ and $\varphi_{K,H}^{\ast} : C(G/H)\to C(G/K)$ are all $G$-equivariant. Therefore, we obtain the following fact.
\begin{lemma}\label{lem_inductive_limit_profinite_completion}
If $G$ is a discrete, residually finite group, then
\[
(C(\G), \sigma) \cong \lim_{\I_G} (C(G/H), \sigma^H).
\]
\end{lemma}

Now if $G$ is a finitely generated group and $n\in \N$ is fixed, then by a theorem of Hall \cite{hall}, $G$ has finitely many subgroups of index $n$. In other words, $\I_G$ is countable. Since $\G$ is a subspace of $\prod_{H \in \I_G} G/H$, it is metrizable. We record this fact for later use.

\begin{lemma}\label{lem_finitely_generated_metrizable}
If $G$ is a discrete, finitely generated, residually finite group, then $\G$ is metrizable, and therefore $C(\G)$ is separable.
\end{lemma}

\subsection{The Topological Join}

Given two compact Hausdorff spaces $X$ and $Y$, the topological join of $X$ and $Y$ is defined as
\[
X\ast Y := ([0,1]\times X \times Y)/\sim
\]
where $\sim$ is the equivalence relation defined by $(0,x,y) \sim (0,x',y)$ and $(1,x,y)\sim (1,x,y')$ for all $x,x'\in X$ and $y,y'\in Y$. Elements of $X\ast Y$ are denoted by the symbol $[t,x,y]$ for the equivalence class containing $(t,x,y)$. \\

Given three compact Hausdorff spaces $X, Y$ and $Z$, we may also define $(X\ast Y)\ast Z$ and $X\ast (Y\ast Z)$ as above. Since all spaces are compact and Hausdorff, these two spaces are naturally homeomorphic, so the join operation is associative. Thus if $X_1, X_2, \ldots, X_n$ are compact Hausdorff spaces, then $X_1\ast X_2\ast \ldots \ast X_n$ may be defined unambiguously. If $X_i = X$ for all $1\leq i\leq n$, we denote the space $X_1\ast X_2\ast \ldots \ast X_n$ by $X^{\ast(n)}$. \\

Now suppose $G$ is a discrete group that acts on both spaces $X$ and $Y$, then there is a natural diagonal action of $G$ on $X\ast Y$ given by $g\cdot [t,x,y] := [t,gx, gy]$. Moreover, this action is free if each individual action is free. In particular, if $G$ acts freely on a compact Hausdorff space $X$ by an action $\gamma : G\curvearrowright X$, then there is an induced action $\gamma^{(n)} : G\curvearrowright X^{\ast(n)}$ that is also free. \\

Our immediate goal is to give an estimate for the covering dimension of the join of two spaces. Henceforth, we write $\dim(Z)$ to denote the covering dimension of a compact Hausdorff space $Z$. Observe that if $X$ is a compact Hausdorff space, then $\mathcal{C}X := X\ast \{p\}$ is the cone over $X$. We denote the points in $\mathcal{C}X$ by $[t,x]$ (by suppressing the $p$).

\begin{prop}\label{prop_dimension_join}
For any two compact Hausdorff spaces $X$ and $Y$,
\[
\dim(X\ast Y) \leq 1 + \dim(X) + \dim(Y).
\]
In particular, if $\dim(X) = 0$, then $\dim(X^{\ast(n)}) \leq n-1$.
\end{prop}
\begin{proof}
With a slight abuse of notation, we write $\mathcal{C}Y := ([0,1]\times Y)/\sim$ where $\sim$ is the equivalence relation $(1,y)\sim (1,y')$ for all $y,y'\in Y$. In other words, $\mathcal{C}X = X\ast \{p\}$ while $\mathcal{C}Y := \{p\}\ast Y$. Let $Z := \mathcal{C}X\times Y \cup X\times \mathcal{C}Y \subset \mathcal{C}X\times \mathcal{C}Y$ and define $g : [0,1]\times X\times Y \to Z$ by
\[
g(t,x,y) := \begin{cases}
([2t,x],y) &: \text{ if } 0\leq t\leq \frac{1}{2} \\
(x, [2t-1,y]) &: \text{ if } \frac{1}{2} \leq t \leq 1.
\end{cases}
\]
Note that $g$ is well-defined and continuous, and descends to a continuous map $f:X\ast Y \to Z$. That $f$ is bijective is easy to check. Since $X\ast Y$ is compact and $Z$ is Hausdorff, $f$ is a homeomorphism. Hence, it suffices to estimate $\dim(Z)$.

For each $n\in \N$, let $A_n = \{[t,x] \in \mathcal{C}X : 1/n\leq t\leq 1, x\in X\}$. Then $A_n$ is homeomorphic to $[1/n,1]\times X$ so $\dim(A_n) \leq \dim(X) + 1$ by the product theorem \cite[Proposition III.2.6]{pears}. Since
\[
\mathcal{C}X = \{p\}\sqcup\left( \bigcup_{n=1}^{\infty} A_n\right),
\]
it follows that $\dim(\mathcal{C}X) \leq \dim(X) + 1$ by \cite[Proposition III.5.3]{pears}. By the product theorem, $\dim(\mathcal{C}X\times Y) \leq 1 + \dim(X) + \dim(Y)$. By symmetry, $\dim(X\times \mathcal{C}Y) \leq 1 + \dim(X) + \dim(Y)$. Moreover, both $A = \mathcal{C}X\times Y$ and $B = X\times \mathcal{C}Y$ are $F_{\sigma}$-sets in $\mathcal{C}X\times \mathcal{C}Y$. By \cite[Proposition III.5.3]{pears}, $\dim(Z) \leq \dim(X) + \dim(Y) + 1$. 
\end{proof}

\subsection{Notational Conventions}
For convenience, we now fix some notation that we will use frequently through the rest of the paper. Henceforth, all groups (denoted $G, H, K$, etc.) will be countable and discrete, and we will write $e$ for the identity of the group. We write $H\lf G$ if $H$ is a subgroup of $G$ of finite index, we write $H\< G$ if $H$ is a normal subgroup of $G$, and we write $H\nf G$ if $H$ is a normal subgroup of finite index.  \\

If $Z$ is a topological space, we write $\dim(Z)$ for its Lebesgue covering dimension. If a group $G$ acts on $Z$ by homeomorphisms, we denote this by $G\curvearrowright Z$ or $\beta : G\curvearrowright Z$ if $\beta : G\to \text{Homeo}(Z)$ denotes the corresponding homomorphism. If $K$ is a set, then we write $F\ssubset K$ if $F$ is a finite subset of $K$. If $A$ is a C*-algebra and $a,b\in A$, we write $[a,b] := ab-ba$, and we write $a\approx_{\epsilon} b$ if $\|a-b\| < \epsilon$. If $A$ is unital, we write $1_A$ for the unit of $A$.

\section{Rokhlin Dimension for actions of Residually Finite groups}\label{sec_definition}

In this section, we recall from \cite{szabo_wu_zacharias} the definition of Rokhlin dimension for actions of residually finite groups. We then give an alternate definition using the profinite completion of the group, and we discuss the notion of Rokhlin dimension with commuting towers. Given a natural number $d\in \N$, we describe explicitly a universal space for actions with Rokhlin dimension $d$ with commuting towers. This last result is a natural analogue of the corresponding result for compact groups due to Gardella \cite[Lemma 4.3]{gardella_compact}. We begin by describing the central sequence algebra relative to a C*-subalgebra, a notion due to Kirchberg \cite{kirchberg}.

\begin{definition}
Given a C*-algebra $A$, let $\ell^{\infty}(\N,A)$ be the space of all bounded sequences in $A$ and $c_0(\N,A)$ be the subspace of sequences that vanish at infinity. If $A_{\infty} := \ell^{\infty}(\N,A)/c_0(\N,A)$, then $A$ embeds in $A_{\infty}$  as the set of all constant sequences so we identify $A$ with its image in $A_{\infty}$. For a C*-subalgebra $D\subset A$, we define 
\begin{eqsplit}
A_\infty\cap D' &= \{ x\in A_{\infty} : xd=dx\text{ for all }d\in D \} \text{ and }\\
\Ann(D,A_\infty) &= \{ x\in A_{\infty} : xd=dx=0\text{ for all }d\in D\}.
\end{eqsplit}
$\Ann(D,A_{\infty})$ is an ideal in $A_{\infty}\cap D'$, so we write
\[
F(D,A) := (A_{\infty} \cap D')/\Ann(D, A_{\infty})
\]
and $\kappa_{D,A} : A_{\infty}\cap D'\to F(D,A)$ for the corresponding quotient map. When $D = A$, we write $F(A)$ for $F(A,A)$ and $\kappa_A$ for $\kappa_{A,A}$. Note that $F(D,A)$ is unital if $D$ is $\sigma$-unital.
\end{definition}

Let $G$ be a discrete group and $\alpha : G\to \Aut(A)$ be an action of $G$ on a C*-algebra $A$. If $D$ is an $\alpha$-invariant subalgebra of $A$, there is a natural induced action of $G$ on $A_{\infty}$ and on $F(D,A)$ which we denote by $\alpha_{\infty}$ and $\widetilde{\alpha}_{\infty}$ respectively. Moreover, there is a $G$-equivariant $\ast$-homomorphism $(F(D,A)\otimes_{\max} D, \widetilde{\alpha}_{\infty}\otimes \alpha\lvert_D) \to (A_{\infty}, \alpha_{\infty})$ given on elementary tensors by $\kappa_{D,A}(x)\otimes a \to x\cdot a$. Under this $\ast$-homomorphism $1_{F(D,A)}\otimes a$ is mapped to $a$ for all $a\in D$, so we think of it as a way to multiply elements of $F(D,A)$ with elements of $D$ to obtain elements of $A_{\infty}$ (in a way that respects the action of $G$). \\

We need one last piece of terminology. Two elements $a$ and $b$ in a C*-algebra are said to be orthogonal (in symbols, we write $a\perp b$) if $ab = ba = a^{\ast}b = b^{\ast}a=0$. A contractive, completely positive map $\varphi : A\to B$ between two C*-algebras is said to have order zero if $\varphi(a)\perp \varphi(b)$ whenever $a\perp b$. As is customary, we will use the abbreviation `c.c.p' for `contractive and completely positive'. \\

With all this in place, we are now in a position to define Rokhlin dimension for actions of residually finite groups. Following \cite{szabo_wu_zacharias}, we first define Rokhlin dimension relative to a subgroup of finite index before defining the full Rokhlin dimension for an action.

\begin{definition}\cite[Definition 5.4]{szabo_wu_zacharias}
Let $A$ be a C*-algebra, $G$ be a discrete, countable group, $H$ be a subgroup of $G$ of finite index, and $\alpha:G\to \Aut(A)$ be an action of $G$ on $A$. We say that $\alpha$ has Rokhlin dimension $d$ relative to $H$ if $d$ is the least integer such that for any separable, $\alpha$-invariant C*-subalgebra $D\subset A$, there exist $(d+1)$ equivariant c.c.p. order zero maps
\[
\varphi_0, \varphi_1, \ldots, \varphi_d :(C(G/H),\sigma^H)\to (F(D,A),\widetilde{\alpha}_{\infty})
\]
such that $\sum_{\ell=0}^{d}\varphi_{\ell}(1_{C(G/H)})=1_{F(D,A)}$. We denote the Rokhlin dimension of $\alpha$ relative to $H$ by $\dr(\alpha, H)$. If no such integer $d$ exists, then we write $\dr(\alpha, H) = +\infty$.
\end{definition}

The following lemma is contained in \cite[Proposition 5.5]{szabo_wu_zacharias}.

\begin{lemma}\label{lem_rokhlin_system_finite_index}
Let $A$ be a C*-algebra, $G$ be a discrete, countable group, $H$ be a subgroup of $G$ of finite index, and $\alpha:G\to \Aut(A)$ be an action of $G$ on $A$. Then $\dr(\alpha, H)\leq d$ if and only if for each $F\ssubset A, M\ssubset G, S\ssubset C(G/H)$ and $\epsilon>0$, there exist $(d+1)$ c.c.p. maps 
\[
\psi_0, \psi_1, \ldots, \psi_d : C(G/H) \to A
\]
satisfying the following properties:
\begin{enumerate}
\item $[\psi_{\ell}(f),a] \approx_{\epsilon} 0$ for all $a\in F, f\in S$ and $0\leq \ell\leq d$.
\item $\psi_{\ell}(\sigma^H_g(f))a \approx_{\epsilon} \alpha_g(\psi_{\ell}(f))a$ for all $a\in F, f\in S, g\in M$ and $0\leq \ell\leq d$.
\item $\psi_{\ell}(f_1)\psi_{\ell}(f_2)a \approx_{\epsilon} 0$ for all $a\in F$ and $f_1, f_2\in S$ such that $f_1\perp f_2$ and for all $0\leq \ell\leq d$.
\item $\sum_{\ell=0}^d \psi_{\ell}(1_{C(G/H)})a \approx_{\epsilon} a$ for all $a\in F$.
\end{enumerate}
\end{lemma}
The set of maps $\{\psi_0, \psi_1, \ldots, \psi_d\}$ satisfying the conditions of this lemma is called an $(H, d, F, M, S, \epsilon)$-Rokhlin system.

\begin{definition}\cite[Definition 5.8]{szabo_wu_zacharias}
Let $A$ be a C*-algebra, $G$ be a residually finite group, and $\alpha:G\to \Aut(A)$ be an action of $G$ on $A$. We define the Rokhlin dimension of $\alpha$ as
\[
\dr(\alpha) = \sup\{ \dr(\alpha, H):  H\lf G\}.
\]
\end{definition}

The next simple proposition will be used repeatedly throughout the rest of the paper. As such, it introduces the profinite completion of $G$ into the picture, thereby allowing us to treat discrete groups and compact groups on an (almost) equal footing.

\begin{prop}\label{prop_rokhlin_system}
Let $A$ be a C*-algebra, $G$ be a residually finite group and $\alpha:G\to \Aut(A)$ be an action of $G$ on $A$. Then $\dr(\alpha)\leq d$ if and only if for any $F\ssubset A, M\ssubset G, S\ssubset C(\G)$ and any $\epsilon>0$ there exist $(d+1)$ c.c.p. maps
\[
\psi_0, \psi_1, \ldots, \psi_d :C(\G)\to A
\]
satisfying the following properties:
\begin{enumerate}
\item $[\psi_{\ell}(f),a]\approx_{\epsilon} 0$ for all $a\in F, f\in S$ and $0\leq \ell\leq d$.
\item $\psi_{\ell}(\sigma_g(f))a \approx_{\epsilon} \alpha_g(\psi_{\ell}(f))a$ for all $a\in F, f\in S, g\in M$ and $0\leq \ell\leq d$.
\item $\psi_{\ell}(f_1)\psi_{\ell}(f_2)a\approx_{\epsilon} 0$ for all $a\in F$ and $f_1, f_2\in S$ such that $f_1\perp f_2$ and all $0\leq \ell\leq d$.
\item $\sum_{\ell=0}^d \psi_{\ell}(1_{C(\bar{G})})a\approx_{\epsilon} a$ for all $a\in F$.
\end{enumerate}
\end{prop}
The set of maps $\{\psi_0, \psi_1, \ldots, \psi_d\}$ satisfying the conditions of this proposition is called a $(d, F,M,S,\epsilon)$-Rokhlin system.
\begin{proof}
Assume that $\dr(\alpha) \leq d$, and let $F\ssubset A, M\ssubset G, S\ssubset C(\G)$ and $\epsilon>0$ be given. By \autoref{lem_inductive_limit_profinite_completion}, $C(\G) \cong \lim_{\I_G} C(G/H)$, so by approximating, we may assume that $S\subset \pi_H^{\ast}(C(G/H))$ for some $H \in \I_G$. Write $S = \pi_H^{\ast}(S')$ for some finite set $S'\subset C(G/H)$ and we may further assume that $1_{C(G/H)} \in S'$, that $e\in M$ and that $\|a\| \leq 1$ for all $a\in F$. By \autoref{lem_rokhlin_system_finite_index}, we obtain $(d+1)$ c.c.p. maps
\[
\widetilde{\psi}_0, \widetilde{\psi}_1, \ldots, \widetilde{\psi}_d : C(G/H) \to A
\]
which form an $(H, d, F, M, S', \epsilon/3)$-Rokhlin system. For each $0\leq \ell\leq d$, since $C(G/H)$ is nuclear, there is a sequence of c.c.p. maps $\rho^i_{\ell} : C(G/H)\to M_{k(i)}(\C)$ and $\theta^i_{\ell} : M_{k(i)}(\C)\to A$ such that
\[
\lim_{i\to \infty} \theta^i_{\ell}\circ \rho^i_{\ell}(f) = \widetilde{\psi}_{\ell}(f)
\]
for all $f\in C(G/H)$. Since $\pi_H^{\ast}$ is injective, we may use Arveson's extension theorem to obtain c.c.p. maps $\overline{\rho}^i_{\ell} : C(\G) \to M_{k(i)}(\C)$ such that $\overline{\rho}^i_{\ell}\circ \pi_H^{\ast} = \rho^i_{\ell}$. Then there exists $i_0\in \N$ such that
\[
\left\|\theta^{i_0}_{\ell}(\overline{\rho}^{i_0}_{\ell}((\pi_H^{\ast}(\sigma^{H}_g(f))))) - \widetilde{\psi}_{\ell}(\sigma_g(f))\right\| < \frac{\epsilon}{3(d+1)}
\]
for all $f\in S'$ and $g\in M$. If $\psi_{\ell} : C(\G) \to A$ is given by $\psi_{\ell} := \theta^{i_0}_{\ell}\circ \overline{\rho}^{i_0}_{\ell}$, then we claim that the maps $\{\psi_0, \psi_1, \ldots, \psi_d\}$ satisfy the required conditions. For the sake of brevity, we verify only condition (4) as the others are similar: If $a\in F$, then
\[
\sum_{\ell=0}^d \psi_{\ell}(1_{C(\G)})a = \sum_{\ell=0}^d \theta^{i_0}_{\ell}\circ \overline{\rho}^{i_0}_{\ell} \circ \pi_H^{\ast}(1_{C(G/H)})a \approx_{\frac{\epsilon}{3}} \sum_{\ell=0}^d \widetilde{\psi}_{\ell}(1_{C(G/H)})a \approx_{\frac{\epsilon}{3}} a.
\]
With the other conditions verified, we conclude that $\{\psi_0, \psi_1, \ldots,\psi_d\}$ form a $(d,F,M,S,\epsilon)$-Rokhlin system. \\

For the converse, choose a subgroup $H\lf G$ and let $\pi_H : \G\to G/H$ be the natural map from \autoref{prop_profinite_completion}. Then $\pi_H^{\ast} : C(G/H)\to C(\G)$ is a unital $G$-equivariant $\ast$-homomorphism. If $F\ssubset A, M\ssubset G, S\ssubset C(G/H)$ and $\epsilon > 0$ are given, then by hypothesis, there exist $(d+1)$ c.c.p. maps
\[
\psi_0, \psi_1, \ldots, \psi_d : C(\G) \to A
\]
satisfying the four conditions listed above for the tuple $(d, F, M, \pi_H^{\ast}(S), \epsilon)$. Then, $\{\psi_{\ell} \circ \pi_H^{\ast} : 0\leq \ell \leq d\}$ clearly forms a $(H,d, F,M,S, \epsilon)$-Rokhlin system as in \autoref{lem_rokhlin_system_finite_index}, proving that $\dr(\alpha, H)\leq d$. This is true for each $H\lf G$, so $\dr(\alpha) \leq d$.
\end{proof}

In order to turn the previous proposition into a global statement (without reference to finite sets and $\epsilon$'s), we need to assume that $G$ is finitely generated to ensure that $C(\G)$ is separable.

\begin{cor}\label{cor_rokhlin_dim_equiv_defn}
Let $A$ be a C*-algebra, $G$ be a finitely generated, residually finite group and $\alpha:G\to \Aut(A)$ be an action of $G$ on $A$. Then $\dr(\alpha) \leq d$ if and only if, for each separable, $\alpha$-invariant C*-subalgebra $D\subset A$, there exist $(d+1)$ equivariant, order zero, c.c.p. maps
\[
\varphi_0, \varphi_1, \ldots, \varphi_d : (C(\G), \sigma) \to (F(D,A),\widetilde{\alpha}_{\infty})
\]
such that $\sum_{\ell=0}^d \varphi_{\ell}(1_{C(\G)}) = 1_{F(D,A)}$.
\end{cor}
\begin{proof}
Suppose that $\dr(\alpha)\leq d$. Choose a separable, $\alpha$-invariant C*-subalgebra $D\subset A$ and choose an increasing sequence of finite sets $(F_n)$ such that $T := \bigcup_{n=1}^{\infty} F_n$ is dense in $D$. Since $\G$ is compact and metrizable, there is an increasing sequence $(S_n)$ of finite subsets of $C(\G)$ such that $\mathcal{A} := \bigcup_{n=1}^{\infty} S_n$ is dense in $C(\G)$. Furthermore, we may arrange it so that for each positive function $f\in C(\G)$ and each $\delta > 0$, there exists $f_0\in \mathcal{A}$ such that $\|f_0-f\| < \delta$ and $\text{supp}(f_0) \subset \text{supp}(f)$ (\cite[Lemma 1.4]{self_rokhlin}). Finally, since $G$ is countable, there is an increasing sequence $(M_n)$ of finite subsets of $G$ such that $G = \bigcup_{n=1}^{\infty} M_n$. For each $n\in \N$, let
\[
\psi^{(n)}_0, \psi^{(n)}_1, \ldots, \psi^{(n)}_d : C(\G)\to A
\]
be a $(d, F_n, M_n, S_n, 1/n)$-Rokhlin system. For $0\leq \ell\leq d$, define $\widetilde{\varphi}_{\ell} : C(\G)\to A_{\infty}$ by
\[
\widetilde{\varphi}_{\ell}(f) := \eta_A((\psi^{(n)}_{\ell}(f)))
\]
where $\eta_A : \ell^{\infty}(\N,A)\to A_{\infty}$ is the natural quotient map. If $f\in \mathcal{A}$ and $x\in T$, then there exists $N\in \N$ such that
\[
\|[\psi^{(n)}_{\ell}(f), x]\| < \frac{1}{n}
\]
for all $n\geq N$. Therefore, $\widetilde{\varphi}_{\ell}(f)$ commutes with $x$ in $A_{\infty}$. Since $\mathcal{A}$ is dense in $C(\G)$ and $T$ is dense in $D$, it follows that $\widetilde{\varphi}_{\ell}(f) \in D'$ for all $f\in C(\G)$. We thus get well-defined maps
\[
\varphi_0, \varphi_1, \ldots, \varphi_d : C(\G) \to F(D,A)
\]
given by $\varphi_{\ell} := \kappa_{D,A}\circ \widetilde{\varphi}_{\ell}$. Following the argument of \cite[Lemma 1.5]{self_rokhlin}, one can show that these maps satisfy the required conditions. \\

For the converse, the proof of \cite[Lemma 3.7]{gardella_rokhlin} applies verbatim, except that $\G$ plays the role of $G$ here.
\end{proof}

In \cite{szabo_wu_zacharias}, the authors introduce a notion that is somewhat finer than Rokhlin dimension, which we now describe. Let $\alpha : G\to \Aut(A)$ be an action of a residually finite group on a C*-algebra $A$. A decreasing sequence $\tau = (G_n)$ of finite index subgroups of $G$ is said to be a regular approximation if for each non-identity element $g\in G$, there exists $n\in \N$ such that $G_n\cap g^GG_n = \emptyset$ where $g^G$ denotes the conjugacy class of $g$ (see \cite[Definition 3.5]{szabo_wu_zacharias}). Given such a sequence $\tau$, the Rokhlin dimension of $\alpha$ with respect to $\tau$ is defined as 
\[
\dr(\alpha, \tau) := \sup_{n \in \N} \dr(\alpha, G_n).
\]
It is clear that $\dr(\alpha, \tau) \leq \dr(\alpha)$ for any regular approximation $\tau$ and there are examples where strict inequality holds. However, if $G$ is a finitely generated group, then there is a regular approximation $\tau$ such that $\dr(\alpha) = \dr(\alpha, \tau)$. Now, the condition $\dr(\alpha, \tau) \leq d$ also admits an equivalent formulation along the lines of \autoref{cor_rokhlin_dim_equiv_defn}. Indeed, the role of $\G$ in that result would now be played by $\G_{\tau} := \varprojlim G/G_n$, and $C(\G_{\tau})$ also carries a natural action of $G$ which we denote by $\sigma_{\tau}$. Since $\G_{\tau}$ is a subspace of $\prod_{n=1}^{\infty} G/G_n$, it is metrizable, so the following result holds even when the group is not finitely generated.

\begin{cor}
Let $A$ be a C*-algebra, $G$ be a residually finite group, $\tau$ be a regular approximation of $G$ and $\alpha : G\to \Aut(A)$ be an action of $G$ on $A$. Then $\dr(\alpha, \tau) \leq d$ if and only if, for each separable, $\alpha$-invariant C*-subalgebra $D\subset A$, there exist $(d+1)$ equivariant, order zero, c.c.p. maps
\[
\varphi_0, \varphi_1, \ldots, \varphi_d : (C(\G_{\tau}), \sigma_{\tau}) \to (F(D,A),\widetilde{\alpha}_{\infty})
\]
such that $\sum_{\ell=0}^d \varphi_{\ell}(1_{C(\G_{\tau})}) = 1_{F(D,A)}$.
\end{cor}

\subsection{Rokhlin Dimension with Commuting Towers}

The notion of Rokhlin dimension with commuting towers has proved to be very useful, particularly in the context of compact groups (see, for instance, \cite{gardella_hirshberg_santiago}). For the purpose of this paper, it is here merely for completeness. Almost all results we prove for Rokhlin dimension with commuting towers have analogous statements for Rokhlin dimension without commuting towers.

\begin{definition}\cite[Definition 5.4]{szabo_wu_zacharias}
Let $A$ be a C*-algebra, $G$ be a discrete, countable group, $H$ be a subgroup of $G$ of finite index, and $\alpha:G\to \Aut(A)$ be an action of $G$ on $A$. We say that $\alpha$ has Rokhlin dimension $d$ with commuting towers relative to $H$ if $d$ is the least integer such that for any separable, $\alpha$-invariant C*-subalgebra $D\subset A$, there exist $(d+1)$ equivariant c.c.p. order zero maps
\[
\varphi_0, \varphi_1, \ldots, \varphi_d :(C(G/H),\sigma^H)\to (F(D,A),\widetilde{\alpha}_{\infty})
\]
with pairwise commuting ranges such that $\sum_{\ell=0}^{d}\varphi_{\ell}(1_{C(G/H)})=1_{F(D,A)}$. We denote the Rokhlin dimension of $\alpha$ with commuting towers relative to $H$ by $\dr^c(\alpha, H)$. If no such integer $d$ exists, then we write $\dr^c(\alpha, H) = +\infty$. We define the Rokhlin dimension of $\alpha$ with commuting towers as
\[
\dr^c(\alpha) = \sup\{ \dr^c(\alpha, H):  H\lf G\}.
\]
\end{definition}

Each result from the previous section has an analogous version for Rokhlin dimension with commuting towers. For brevity, we state (without proof) the two that we will need in the future.

\begin{prop}\label{prop_rokhlin_system_commuting}
Let $A$ be a C*-algebra, $G$ be a residually finite group and $\alpha:G\to \Aut(A)$ be an action of $G$ on $A$. Then $\dr^c(\alpha)\leq d$ if and only if for any $F\ssubset A, M\ssubset G, S\ssubset C(\G)$ and any $\epsilon>0$ there exist $(d+1)$ c.c.p. maps
\[
\psi_0, \psi_1, \ldots, \psi_d :C(\G)\to A
\]
satisfying the following properties:
\begin{enumerate}
\item $[\psi_{\ell}(f),a]\approx_{\epsilon} 0$ for all $a\in F, f\in S$ and $0\leq \ell\leq d$.
\item $\psi_{\ell}(\sigma_g(f))a \approx_{\epsilon} \alpha_g(\psi_{\ell}(f))a$ for all $a\in F, f\in S, g\in M$ and $0\leq \ell\leq d$.
\item $\psi_{\ell}(f_1)\psi_{\ell}(f_2)a\approx_{\epsilon} 0$ for all $a\in F$ and $f_1, f_2\in S$ such that $f_1\perp f_2$ and all $0\leq \ell\leq d$.
\item $\sum_{\ell=0}^d \psi_{\ell}(1_{C(\bar{G})})a\approx_{\epsilon} a$ for all $a\in F$.
\item $[\psi_k(f_1), \psi_{\ell}(f_2)]a \approx_{\epsilon} 0$ for all $f_1, f_2\in S, 0\leq k, \ell\leq d$ and all $a\in F$.
\end{enumerate}
\end{prop}

The set of maps $\{\psi_0, \psi_1, \ldots, \psi_d\}$ satisfying the conditions of this proposition is called a $(d, F,M,S,\epsilon)$-commuting Rokhlin system.

\begin{prop}
Let $A$ be a C*-algebra, $G$ be a finitely generated, residually finite group and $\alpha:G\to \Aut(A)$ be an action of $G$ on $A$. Then $\dr^c(\alpha) \leq d$ if and only if, for each separable, $\alpha$-invariant C*-subalgebra $D\subset A$, there exist $(d+1)$ equivariant, order zero, c.c.p. maps
\[
\varphi_0, \varphi_1, \ldots, \varphi_d : (C(\G), \sigma) \to (F(D,A),\widetilde{\alpha}_{\infty})
\]
with pairwise commuting ranges such that $\sum_{\ell=0}^d \varphi_{\ell}(1_{C(\G)}) = 1_{F(D,A)}$.
\end{prop}

The reason we have isolated this result is that it leads directly to the following theorem. This was proved in the context of compact groups by Gardella \cite[Lemma 4.3]{gardella_rokhlin}. Since we wish to identify the universal space explicitly (and also estimate its covering dimension), we repeat the proof below with appropriate modifications. Recall that if $G$ is a finitely generated, residually finite group, then $\G$ is a compact metric space. For $n\in \N$, we write $\G^{\ast(n)}$ for the $n$-fold join of $\G$ with itself. Moreover, if $\sigma : G\to \Aut(C(\G))$ is the natural action described earlier, we write $ \sigma^{(n)} : G\to \Aut(C(\G^{\ast(n)}))$ for the induced action of $G$ on $C(\G^{\ast(n)})$.

\begin{theorem}
Let $A$ be a C*-algebra, $G$ be a finitely generated, residually finite group and $\alpha:G\to \Aut(A)$ be an action of $G$ on $A$. Then, $\dr^c(\alpha) \leq d$ if and only if, for each separable, $\alpha$-invariant C*-subalgebra $D\subset A$, there is a unital, $G$-equivariant $\ast$-homomorphism
\[
\varphi : (C(\G^{\ast(d+1)}), \sigma^{(d+1)}) \to (F(D,A), \widetilde{\alpha}_{\infty}).
\]
Moreover, $\dim(\G^{\ast(d+1)}) \leq d$.
\end{theorem}
\begin{proof}
Suppose that $\dr^c(\alpha)\leq d$ and let $D\subset A$ be a separable, $\alpha$-invariant C*-subalgebra of $A$. Let $\varphi_0, \varphi_1, \ldots, \varphi_d : (C(\overline{G}),\sigma) \to (F(D,A),\widetilde{\alpha}_{\infty})$ be $G$-equivariant, order zero c.c.p. maps with pairwise commuting ranges such that $\sum_{\ell=0}^d \varphi_{\ell}(1_{C(\G)}) = 1_{F(D,A)}$. We then follow the line of reasoning from \cite[Lemma 4.3]{gardella_compact}. Let $t \in C_0(0,1]$ denote the identity function. Then for each $0\leq \ell\leq d$, there exists a $\ast$-homomorphism $\rho_{\ell} : C_0(0,1]\otimes C(\G)\to F(D,A)$ such that
\[
\rho_{\ell}(t\otimes f) = \varphi_{\ell}(f)
\]
for all $f\in C(\G)$. Let $\mathcal{C}\G$ denote the cone on $\G$, whose elements we denote by $[t,x]$ for $t\in [0,1]$ and $x\in \G$. Note that $C(\mathcal{C}\G)$ is the minimal unitization of $C_0(0,1]\otimes C(\G)$. In other words, $C(\mathcal{C}\overline{G}) = \{f : [0,1]\to C(\G) \text{ continuous}: f(0) \in \C1_{C(\G)}\}$. Then $G$ acts on $C(\mathcal{C}\G)$ by  $\gamma : G\to \text{Aut}(C(\mathcal{C}\G))$ given by $\gamma_g\cdot f ([t,x]) := f([t,\beta_{g^{-1}}(x)])$. Each $\rho_{\ell}$ is $G$-equivariant as well. Let
\[
E := \bigotimes_{\ell=0}^d C(\mathcal{C}\G)
\]
and let $\delta : G\to \Aut(E)$ be the action induced by $\gamma$. Let $\omega : E\to \C$ denote the $\ast$-homomorphism give on simple tensors by $\omega(f_0\otimes f_1\otimes \ldots \otimes f_d) := \prod_{\ell=0}^d f_{\ell}(0)$, and let $J := \ker(\omega)$. Observe that
\[
J = \{f : (\mathcal{C}\G)^{d+1} \to \C \text{ continuous} : f([0,x_0], [0, x_1], \ldots, [0, x_d]) = 0 \text{ for all } x_i \in \G\}
\]
By \cite[Lemma 5.2]{hirshberg}, there is a unique $\ast$-homomorphism $\rho : J\to F(D,A)$ induced by the tuple $(\rho_0, \ldots, \rho_d)$. Also note that $\delta_g(J)\subset J$ for each $g\in G$, so we get an induced action $\delta : G\to \text{Aut}(J)$ which is realized as
\[
\delta_g(f)([s_0,x_0], [s_1, x_1], \ldots, [s_d, x_d]) := f([s_0, \beta_{g^{-1}}(x_0)], \ldots, [s_d, \beta_{g^{-1}}(x_d)]). 
\]
Let $e \in J$ denote the function $e([s_0,x_0], \ldots, [s_d, x_d]) := \sum_{\ell=0}^d s_{\ell}$. Then $\delta_g(e) = e$ for all $g\in G$, and $\rho(e) = \sum_{\ell=0}^d \varphi_{\ell}(1_{C(\overline{G})}) = 1_{F(D,A)}$. If $I$ denotes the $\delta$-invariant ideal in $J$ generated by the set $\{ef - f : f\in J\}$ and $C := J/I$, then $\rho$ induces a unital, $G$-equivariant $\ast$-homomorphism
\[
\overline{\rho} : C \to (F(D,A),\widetilde{\alpha}_{\infty}).
\] 
Now observe that $C = C(Y)$ where
\[
Y = \left\{([s_0,x_0], [s_1, x_1], \ldots, [s_d, x_d]) \in (\mathcal{C} \G)^{d+1}: \sum_{\ell=0}^d s_{\ell} =  1\right\}
\]
with the action of $G$ on $Y$ given by $g\cdot ([s_0, x_0], \ldots, [s_d, x_d]) := ([s_0, \beta_g(x_0)], \ldots, [s_d, \beta_g(x_d)])$. Hence there is a $G$-equivariant homeomorphism $Y \cong \overline{G}^{\ast(d+1)}$ and the result is proved. \\

The converse follows by simply reversing the above argument. Finally, observe that $\dim(\G) = 0$ because $\G$ is totally disconnected, so the inequality $\dim(\G^{\ast(d+1)}) \leq d$ follows from \autoref{prop_dimension_join}.
\end{proof}

\section{Permanence Properties}\label{sec_permanence}
We now wish to prove some permanence properties for actions with finite Rokhlin dimension along the lines of \cite[Section 3]{gardella_compact} and \cite{hirshberg_phillips}. The next lemma is stated as \cite[Lemma 3.3]{hirshberg_phillips}. We give a slightly more detailed proof here because we will seek to generalize it later in the section. 

\begin{lemma}\label{lem_positive_G_invariant_hereditary}
Let $A$ be a C*-algebra, $G$ be an amenable group, and $\beta : G\to \text{Aut}(A)$ be an action of $G$ on $A$. Let $D \subset A$ be a $\beta$-invariant hereditary subalgebra and let $\alpha : G\to \text{Aut}(D)$ denote the restriction of $\beta$ to $D$. Then for any $M \ssubset G, F\ssubset D$ and any $\eta > 0$, there exists a positive contraction $q\in D$ satisfying the following conditions:
\begin{enumerate}
\item $\|\alpha_s(q) - q\| < \eta$ for all $s\in M$.
\item $\|qa - a\| < \eta, \|aq-a\| < \eta$ and $\|qa - aq\| < \eta$ for all $a\in F$.
\end{enumerate}
\end{lemma}
\begin{proof}
Since the sets in question are finite, we may assume that $\|a\| \leq 1$ for all $a\in F$. Let $(e_{\lambda})_{\lambda \in \Lambda} \subset D$ be an approximate unit in $D$. Let $K \subset G$ be a F\o lner set such that $M\subset K$ and
\[
\frac{|sK\Delta K|}{|K|} < \eta
\]
for all $s\in M$. Define
\[
q_{\lambda} := \frac{1}{|K|}\sum_{t\in K} \alpha_t(e_{\lambda}).
\]
Observe that $(q_{\lambda})_{\lambda \in \Lambda}$ is an approximate unit for $D$ with the property that $\|\alpha_s(q_{\lambda}) - q_{\lambda}\| < \eta$ for all $s\in M$. Then $q := q_{\lambda}$ satisfies the required conditions for $\lambda$ large enough.
\end{proof}

Notice that this is the first time we have needed to assume that $G$ is an amenable group. Indeed, amenability is a crucial assumption in many of the results to follow. We should mention that amenability is implicitly an assumption in many of the results of \cite{szabo_wu_zacharias} as well: There, one would like groups to have box spaces with finite asymptotic dimension for a variety of results to be useful, and this condition automatically implies amenability (see \cite[Remark 3.12]{szabo_wu_zacharias}). We now prove a number of permanence properties analogous to those proved by Gardella \cite[Theorem 3.8]{gardella_compact} for compact groups.

\begin{theorem}\label{thm_permanence}
Let $A$ be a C*-algebra, $G$ be an amenable, residually finite group and $\alpha: G\to\ \text{Aut}(A)$ be an action of $G$ on $A$. 
\begin{enumerate}
\item Let $B$ be a C*-algebra and $\beta : G\to \text{Aut}(B)$ be an action. Let $A\otimes B$ be any C*-algebra completion of $A\odot B$ for which the tensor product action $g \mapsto \alpha_g\otimes \beta_g$ is defined and continuous. Then
\begin{eqsplit}
\dr(\alpha\otimes \beta) &\leq \min\{\dr(\alpha), \dr(\beta)\} \text{ and } \\
\dr^c(\alpha\otimes \beta) &\leq \min\{\dr^c(\alpha), \dr^c(\beta)\}
\end{eqsplit}
\item Let $I$ be an $\alpha$-invariant ideal in $A$, and let $\overline{\alpha} \in \text{Aut}(A/I)$ be the induced action on the quotient. Then
\begin{eqsplit}
\dr(\overline{\alpha}) &\leq \dr(\alpha) \text{ and }\\
\dr^c(\overline{\alpha}) &\leq \dr^c(\alpha)
\end{eqsplit}
\item Let $D \subset A$ an $\alpha$-invariant hereditary subalgebra of $A$ and let $\alpha\lvert_D : G\to \text{Aut}(D)$ be the induced action on $D$. Then,
\begin{eqsplit}
\dr(\alpha\lvert_D) &\leq \dr(\alpha) \text{ and }  \\
\dr^c(\alpha\lvert_D) &\leq \dr^c(\alpha).
\end{eqsplit}
\item Let $(A_n, \iota_n)$ be a direct system of C*-algebras, and let $\alpha^{(n)} \in \text{Aut}(A_n)$ be automorphisms such that $\iota_n \circ \alpha^{(n)} = \alpha^{(n+1)}$ for all $n \in \N$. Let $\alpha = \lim_{n\to \infty} \alpha^{(n)}$ denote the induced action on $A := \lim_{n\to \infty} (A_n, \iota_n)$. Then,
\begin{eqsplit}
\dr(\alpha) &\leq \liminf \dr(\alpha^{(n)}) \text{ and} \\
\dr^c(\alpha) &\leq \liminf \dr^c(\alpha^{(n)}).
\end{eqsplit}
\end{enumerate}
\end{theorem}
\begin{proof}
In each case, the argument for $\dr^c(\cdot)$ is similar to that of $\dr(\cdot)$, so we omit the former.
~\begin{enumerate}
\item This proof is, in principle, identical to that of \cite[Theorem 3.8]{gardella_compact}, but we no longer need to assume that $B$ is unital. We first assume that $d:= \dr(\beta) < \infty$. Let $F\ssubset A\otimes B, M\ssubset G, S\ssubset C(\G)$ and let $\epsilon >0$ be given. We fix $\eta > 0$ to be chosen later. By approximating, we may assume that $F \subset A\odot B$. By further decomposing, we may assume that $F$ consists of elementary tensors, say $F = \{a_i\otimes b_i : 1\leq i\leq n\}$. Let $F_A := \{a_i : 1\leq i\leq n\}$ and $F_B := \{b_i : 1\leq i\leq n\}$. Then there are $(d+1)$ c.c.p. maps
\[
\varphi_0, \varphi_1, \ldots, \varphi_d : C(\overline{G})\to B
\]
which forms a $(d, F_B, M, S, \eta)$-Rokhlin system. Choose $q \in A$ by \autoref{lem_positive_G_invariant_hereditary} corresponding to the triple $(F_A, M, \eta)$. Then define $\theta : B\to A\otimes B$ by $b \mapsto q\otimes b$, and let $\psi_{\ell} : C(\G)\to A\otimes B$ be the c.c.p. map given by $\psi_{\ell} := \theta\circ \varphi_{\ell}$ for $0\leq \ell\leq d$. To verify that $\{\psi_0, \ldots, \psi_d\}$ forms a $(d, F, M, S, \epsilon)$-Rokhlin system, we must check the four conditions of \autoref{prop_rokhlin_system}. Since the arguments are similar, we only verify condition (2): If $g\in M, f \in S, x = a_i\otimes b_i \in F$ and $0\leq \ell\leq d$, then
\begin{eqsplit}
(\alpha\otimes \beta)_g(\psi_{\ell}(f))x &= \alpha_g(q) a_i \otimes \beta_g(\varphi_{\ell}(f))b_i \\
&\approx_{\eta} qa_i\otimes \beta_g(\varphi_{\ell}(f))b_i \\
&\approx_{\eta} qa_i\otimes \varphi_{\ell}(\sigma_g(f))b_i = \psi_{\ell}(\sigma_g(f))x.
\end{eqsplit}
We may verify the other conditions in a similar manner. One sees that if $\eta := \epsilon/(d+2)$, then the maps $\{\psi_0, \ldots, \psi_d\}$ forms a $(d, F, M, S, \epsilon)$-Rokhlin system. We conclude that $\dr(\alpha\otimes \beta) \leq \dr(\beta)$. The argument is symmetric, so it follows that
\[
\dr(\alpha\otimes \beta) \leq \min\{\dr(\alpha), \dr(\beta)\}.
\]
\item The argument here is identical to that of \cite[Proposition 3.4]{hirshberg_phillips}. Since it is omitted there, we give a sketch here. Assume that $d := \dr(\alpha) < \infty$. Let $F\ssubset A/I, S\ssubset C(\G), M\ssubset G$ and $\epsilon > 0$ be given. Let $\widetilde{F}\ssubset A$ be a finite set of preimages of elements in $F$ under the quotient map $Q : A\to A/I$. We may arrange it so that the norms of the lifts are the same as the norms of their images in $F$. By \autoref{prop_rokhlin_system}, we may choose c.c.p. maps
\[
\psi_0, \psi_1, \ldots, \psi_d : C(\G) \to A
\]
which form a $(d, \widetilde{F}, M, S, \epsilon)$-Rokhlin system. It is then clear that if $\varphi_{\ell} := Q\circ \psi_{\ell}$, then $\{\varphi_0, \varphi_1, \ldots, \varphi_d\}$ forms a $(d, F, M, S, \epsilon)$-Rokhlin system for the action $\overline{\alpha}$. Hence, $\dr(\overline{\alpha}) \leq d$. 
\item The argument is identical to \cite[Theorem 3.8]{gardella_compact} with $\G$ playing the role of $G$ (one also needs to appeal to \autoref{lem_positive_G_invariant_hereditary} to obtain the positive element $x$ in that proof).
\item Again, the argument is identical.
\end{enumerate}
\end{proof}

In \cite[Theorem 2.10]{hirshberg_phillips}, the authors have shown that if $\alpha : G\to \Aut(A)$ is an action of a finite group $G$ on a C*-algebra $A$ and if $J$ is an $\alpha$-invariant ideal of $A$, then $\dr^c(\alpha) \leq \dr^c(\alpha\lvert_J) + \dr^c(\overline{\alpha}) + 1$ where $\alpha\lvert_J : G\to \Aut(J)$ and $\overline{\alpha} : G\to \Aut(A/J)$ are the natural actions induced by $\alpha$ on $J$ and $A/J$ respectively. We wish to prove that the same estimates also hold for actions of amenable, residually finite groups. The next lemma isolates one part of the proof, and may be thought of along the lines of \autoref{lem_positive_G_invariant_hereditary}.

\begin{lemma}\label{lem_extension_approx_unit}
Let $G$ be an amenable group, and let
\[
0 \to (J,\alpha) \xrightarrow{\iota} (A,\beta)\xrightarrow{\pi} (B,\gamma) \to 0
\]
be a short exact sequence of $G$-algebras. If $d\in \N, F \ssubset A, M\ssubset G$ and $\eta > 0$, then there exists $q\in J$ such that $0\leq q\leq 1$ and if $\delta := \eta/(d+6)$, then the following conditions are satisfied:
\begin{enumerate}
\item If $a, b_1,b_2 \in F$ are such that $\pi(b_1b_2a)\approx_{\delta} 0$, then
\[
(1-q)^{1/2}b_1(1-q)b_2(1-q)^{1/2}a\approx_{\eta} 0
\]
\item If $a, b_0, b_1, \ldots, b_d \in F$ are such that $\sum_{j=0}^d \pi(b_ja) \approx_{\delta} \pi(a)$, then
\[
(1-q)^{1/2}(b_0 + b_1 + \ldots + b_d)(1-q)^{1/2}a \approx_{\eta} (1-q)a
\]
\item If $a, b\in F$ are such that $[\pi(b), \pi(a)] \approx_{\delta} 0$, then
\[
[(1-q)^{1/2}b(1-q)^{1/2}, a] \approx_{\eta} 0
\]
\item If $a, b_1, b_2\in F$ are such that $[\pi(b_1),\pi(b_2)]\pi(a) \approx_{\delta} 0$, then
\[
[(1-q)^{1/2}b_1(1-q)^{1/2}, (1-q)^{1/2}b_2(1-q)^{1/2}]a \approx_{\eta} 0
\]
\item If $b_1, b_2 \in F$ are such that $\pi(b_1)\approx_{\delta} \pi(b_2)$, then
\[
(1-q)^{1/2}b_1(1-q)^{1/2} \approx_{\eta} (1-q)^{1/2}b_2(1-q)^{1/2}
\]
\item  For every $b\in F$,
\[
[b, q] \approx_{\eta} 0, [b,q^{1/2}] \approx_{\eta} 0 \text{ and } [b, (1-q)^{1/2}] \approx_{\eta} 0.
\]
\item $\alpha_s(q) \approx_{\eta} q$ whenever $s\in M$.
\end{enumerate}
\end{lemma}
\begin{proof}
Let $(e_{\lambda})_{\lambda \in \Lambda}$ be a quasi-central approximate unit for $J$. By the argument in \autoref{lem_positive_G_invariant_hereditary}, we may assume that $\|\alpha_s(e_{\lambda}) - e_{\lambda}\| < \eta$ for all $\lambda \in \Lambda$ and all $s\in M$. Let $\delta := \eta/(d+6)$ and assume without loss generality that $\|b\| \leq 1$ for all $b\in F$. Replacing $A$ by its unitalization, we may assume that $A$ is unital. Let $f:[0,1]\to \R$ and $g:[0,1]\to \R$ be the functions $f(t) := t^{1/2}$ and $g(t) := (1-t)^{1/2}$. By \cite[Exercise 3.9.6]{higson_roe}, there is $\rho > 0$ such that $0 < \rho < \delta$ and for any $b\in F$ and $0 \leq x\leq 1$,
\[
\|f(x)b - bf(x)\| < \delta \text{ and } \|g(x)b - bg(x)\| < \delta.
\]
hold whenever $\|xb - bx\| < \rho$. Replacing $(e_{\lambda})$ by a subnet, we assume that $\|e_{\lambda}b - be_{\lambda}\| < \rho$ for all $b\in F$ and all $\lambda \in \Lambda$. We wish to choose $q := e_{\lambda}$ for $\lambda$ large enough. To do this, we make repeated use of the fact that if $c\in A$ is such that $\|\pi(c)\| < \delta$, then there exists $\lambda' \in \Lambda$ such that $\|c(1-e_{\lambda})\| < \delta$ for all $\lambda \geq \lambda'$.
\begin{enumerate}
\item If $a, b_1, b_2\in F$ are such that $\pi(b_1b_2a) \approx_{\delta} 0$, then by the previous remark, there exists $\lambda_0 \in \Lambda$ such that if $q := e_{\lambda}$ for any $\lambda \geq \lambda_0$, then $b_1(1-q)b_2(1-q)a \approx_{\delta} 0$. In that case,
\[
(1-q)^{1/2}b_1(1-q)b_2(1-q)^{1/2}a \approx_{2\delta} b_1(1-q)b_2(1-q)a \approx_{\delta} 0.
\]
\item If $a, b_0, b_1, \ldots, b_d\in F$ are such that $\sum_{j=0}^d \pi(b_ja) \approx_{\delta} \pi(a)$, then there exists $\lambda_1 \in \Lambda$ such that if $q := e_{\lambda}$ for any $\lambda \geq \lambda_1$, then
\[
\sum_{j=0}^d b_j(1-q)a \approx_{\delta} (1-q)a.
\]
In that case,
\[
(1-q)^{1/2}\left(\sum_{j=0}^d b_j\right)(1-q)^{1/2}a \approx_{(d+1)\delta} \left(\sum_{j=0}^d b_j\right)(1-q)a \approx_{\delta} (1-q)a
\]
\item If $a,b\in F$ are such that $[\pi(b), \pi(a)] \approx_{\delta} 0$, then there exists $\lambda_2 \in \Lambda$ such that if $q := e_{\lambda}$ for any $\lambda \geq \lambda_2$, then $[b(1-q),a] \approx_{\delta} 0$. Hence,
\[
[(1-q)^{1/2}b(1-q)^{1/2}, a] \approx_{2\delta} [b(1-q),a] \approx_{\delta} 0.
\]
\item If $a, b_1, b_2 \in F$ are such that $[\pi(b_1),\pi(b_2)]\pi(a) \approx_{\delta} 0$, then there exists $\lambda_3 \in \Lambda$ such that if $q := e_{\lambda}$ for any $\lambda \geq \lambda_3$, then $[b_1(1-q), b_2(1-q)]a \approx_{\delta} 0$. Hence,
\[
[(1-q)^{1/2}b_1(1-q)^{1/2}, (1-q)^{1/2}b_2(1-q)^{1/2}]a \approx_{4\delta} [b_1(1-q),b_2(1-q)] \approx_{\delta} 0.
\]
\item If $b_1, b_2 \in F$ are such that $\pi(b_1) \approx_{\delta} \pi(b_2)$, then there exists $\lambda_4 \in \Lambda$ such that if $q := e_{\lambda}$ for any $\lambda \geq \lambda_4$, then $b_1(1-q) \approx_{\delta} b_2(1-q)$. Thus,
\[
(1-q)^{1/2}b_1(1-q)^{1/2} \approx_{\delta} b_1(1-q) \approx_{\delta} b_2(1-q) \approx_{\delta} (1-q)^{1/2}b_2(1-q)^{1/2}.
\]
\item If $b\in F$, then if $q := e_{\lambda}$ for any $\lambda \in \Lambda$, we have $[b,e_{\lambda}] \approx_{\delta} 0$ by construction. Moreover, $[b,e_{\lambda}^{1/2}] \approx_{\delta} 0$ and $[b, (1-e_{\lambda})^{1/2}] \approx_{\delta} 0$ hold because of our choice of $\rho$.
\item Once again, if $q := e_{\lambda}$ for any $\lambda \in \Lambda$, then we have $\alpha_s(q)\approx_{\eta} q$ for any $s\in M$.
\end{enumerate}
Therefore, if we choose $\lambda \in \Lambda$ such that $\lambda \geq \lambda_j$ for all $0\leq j\leq 4$, then $q := e_{\lambda}$ satisfies the required conditions.
\end{proof}

\begin{theorem}\label{thm_extensions}
Let $A$ be a C*-algebra and $G$ be an amenable, finitely generated, residually finite group. Let
\[
0 \to (J,\alpha) \xrightarrow{\iota} (A,\beta)\xrightarrow{\pi} (B,\gamma) \to 0
\]
be a short exact sequence of $G$-algebras. Then
\begin{eqsplit}
\dr(\beta) &\leq \dr(\alpha) + \dr(\gamma) + 1, \text{ and } \\
\dr^c(\beta) &\leq \dr^c(\alpha) + \dr^c(\gamma) + 1
\end{eqsplit}
\end{theorem}
\begin{proof}
We prove the second inequality as the proof of the first inequality is similar (indeed, it is subsumed in the former). Assume that $d_J := \dr^c(\alpha) < \infty$ and $d_B := \dr^c(\gamma) < \infty$. Fix finite sets $F \ssubset A, S\ssubset C(\G), M\ssubset G$ and $\epsilon > 0$. Moreover, we assume without loss of generality that $e\in M$, $\|a\|\leq 1$ for all $a\in F$, $\|f\|\leq 1$ for all $f\in S$ and that $1_{C(\G)}\in S$. Let $\eta > 0$ be fixed to be chosen later.

By hypothesis, there exist $(d_B+1)$ c.c.p. maps
\[
\overline{\psi}_0, \overline{\psi}_1, \ldots, \overline{\psi}_{d_B} : C(\G)\to B
\]
that form a $(d_B, \pi(F), M, S, \eta/(d_B+6))$-commuting Rokhlin system. Since $G$ is finitely generated, $C(\G)$ is both separable and nuclear. By the Choi-Effros theorem, there exist c.c.p. maps $\widetilde{\psi}_0, \widetilde{\psi}_1, \ldots, \widetilde{\psi}_{d_B} : C(\G) \to A$ such that $\pi\circ \widetilde{\psi}_k = \overline{\psi}_k$ for all $0\leq k\leq d_B$. Let
\[
F' := F \cup\{\widetilde{\psi}_k(\sigma_g(f)) : f\in S, g\in M, 0\leq k\leq d_B\}
\]
and choose an element $q \in J$ satisfying the conditions of \autoref{lem_extension_approx_unit} for the tuple $(d_B, F', M, \eta)$. Define $\psi_0, \psi_1, \ldots, \psi_{d_B} : C(\G)\to A$ by
\[
\psi_k(f) := (1-q)^{1/2}\widetilde{\psi}_k(f)(1-q)^{1/2}
\]
Then each $\psi_k$ is a c.c.p. map. Moreover, by \autoref{lem_extension_approx_unit}, we see that
\begin{enumerate}
\item $\psi_k(f_1)\psi_k(f_2)a \approx_{\eta} 0$ for all $a\in F$ and $f_1, f_2 \in S$ such that $f_1\perp f_2$ and for all $0\leq k\leq d_B$.
\item $\sum_{k=0}^{d_B} \psi_k(1_{C(\G)})a \approx_{\eta} (1-q)a$ for any $a\in F$.
\item $[\psi_k(f), a] \approx_{\eta} a$ for all $a\in F, f\in S$ and $0\leq k\leq d_B$.
\item $[\psi_j(f_1),\psi_k(f_2)]a \approx_{\eta} 0$ for any $f_1, f_2 \in S, 0\leq j,k\leq d_B$ and any $a\in F$.
\item If $f\in S, a \in F$ and $g\in M$ and $0\leq k\leq d_B$, then
\begin{eqsplit}
\beta_g(\psi_k(f))a &\approx_{\eta} \beta_g(\widetilde{\psi}_k(f)(1-q))a \\
&\approx_{2\eta} \widetilde{\psi}_k(\sigma_g(f))(1-q)a \\
&\approx_{\eta} (1-q)^{1/2}\widetilde{\psi}_k(\sigma_g(f))(1-q)^{1/2}a \\
&= \psi_k(\sigma_g(f))a
\end{eqsplit}
\end{enumerate}
Now set
\[
F_J := \{q, q^{1/2}\}\cup\{qa : a\in F\}\cup\{\widetilde{\psi}_k(f)q : f\in S, 0\leq k\leq d_B\} \ssubset J,
\]
and $S_J := \{\sigma_g(f) : f\in S, g\in M\} \ssubset C(\G)$. Then, there exist $(d_J+1)$ c.c.p. maps
\[
\widetilde{\varphi}_0, \widetilde{\varphi}_1, \ldots, \widetilde{\varphi}_{d_J} : C(\G) \to J
\]
which form a $(d_J, F_J, M, S_J, \eta)$-commuting Rokhlin system. Define $\varphi_0, \varphi_1, \ldots, \varphi_{d_J} : C(\G)\to A$ by
\[
\varphi_j(f) := q^{1/2}\widetilde{\varphi}_j(f)q^{1/2}.
\]
Then each $\varphi_j$ is a c.c.p. map. Moreover, the following hold:
\begin{enumerate}
\item If $f_1, f_2 \in S$ are such that $f_1\perp f_2$ and $a\in F$, then
\[
\varphi_j(f_1)\varphi_j(f_2)a \approx_{\eta} q^{3/2}\widetilde{\varphi}_j(f_1)\widetilde{\varphi}_j(f_2)q^{1/2}a \approx_{\eta} 0.
\]
\item If $a\in F$, then
\[
\sum_{j=0}^{d_J} \varphi_j(1_{C(\G)})a \approx_{(d_J+1)\eta} \sum_{j=0}^{d_J} \widetilde{\varphi}_j(1_{C(\G)})qa \approx_{\eta} qa.
\]
\item If $f\in S$ and $a\in F$, then
\[
[\varphi_j(f), a] \approx_{2\eta} \widetilde{\varphi}_j(f)qa - aq\widetilde{\varphi}_j(f) \approx_{\eta} [\widetilde{\varphi}_j(f), qa] \approx_{\eta} 0.
\]
\item If $g\in M, f\in S$ and $a\in F$, then
\begin{eqsplit}
\alpha_g(\varphi_j(f))a &\approx_{\eta} \alpha_g(\widetilde{\varphi}_j(f)q)a \\
&\approx_{\eta} \alpha_g(\widetilde{\varphi}_j(f))qa \\
&\approx_{\eta} \widetilde{\varphi}_j(\sigma_g(f))qa \\
&\approx_{\eta} q^{1/2}\widetilde{\varphi}_j(\sigma_g(f))q^{1/2}a = \varphi_j(\sigma_g(f))a.
\end{eqsplit}
\item If $f_1, f_2\in S$ and $0\leq k,j\leq d_J$ and $a\in F$, then
\[
[\varphi_j(f_1),\varphi_k(f_2)]a \approx_{6\eta} [\widetilde{\varphi}_j(f_1),\widetilde{\varphi}_k(f_2)]q^2a \approx_{\eta} 0
\]
\item Finally, if $f_1, f_2\in S$ and $0\leq j\leq d_J, 0\leq k\leq d_B$ and $a\in F$, then
\begin{eqsplit}
[\varphi_j(f_1),\psi_k(f_2)]a &\approx_{6\eta} \widetilde{\varphi}_j(f_1)\widetilde{\psi}_k(f_2)q(1-q)a - \widetilde{\psi}_k(f_2)\widetilde{\varphi}_j(f_1)q(1-q)a  \\
&\approx_{\eta} \widetilde{\varphi}_j(f_1)\widetilde{\psi}_k(f_2)q(1-q)a - \widetilde{\psi}_k(f_2)q\widetilde{\varphi}_j(f_1)(1-q)a \\
&= [\widetilde{\varphi}_j(f_1), \widetilde{\psi}_k(f_2)q](1-q)a \\
&\approx_{\eta} 0.
\end{eqsplit}
\end{enumerate}
Thus, if $a\in F$, then
\[
\left(\sum_{k=0}^{d_B} \psi_k(1_{C(\G)}) + \sum_{j=0}^{d_J} \varphi_j(1_{C(\G)})\right)a \approx_{(d_J+2)\eta} a.
\]
Therefore, if $\eta > 0$ is chosen as $\eta := \epsilon/(d_J+8)$, then the system $\{\psi_0, \psi_1, \ldots, \psi_{d_B}, \varphi_0, \varphi_1, \ldots, \varphi_{d_J}\}$ forms a $(d_B+d_J+1, F, M, S, \epsilon)$-commuting Rokhlin system for the action $\beta$. By \autoref{prop_rokhlin_system_commuting}, we conclude that $\dr^c(\beta) \leq d_B + d_J + 1$.
\end{proof}

In the above theorem, we assume that $G$ was finitely generated to ensure that $C(\G)$ is separable so that the Choi-Effros theorem may be used. However, if we use the \autoref{lem_rokhlin_system_finite_index} (or its analogue for commuting towers) instead of \autoref{prop_rokhlin_system_commuting}, we can avoid this requirement. We have presented this proof here to highlight the role of $\G$ and also because the other proof is notationally even more cumbersome. Moreover, this proof also works mutatis mutandis for second countable compact groups, thus answering a question of Gardella \cite[Question 5.1]{gardella_compact}. \\

From the arguments given in both \autoref{thm_permanence} and \autoref{thm_extensions}, it is evident that many results that are true for actions of compact groups are also true for discrete, residually finite groups. Indeed, for a discrete group $G$, the profinite completion $\G$ plays the same role that the group does in the compact case, thanks to \autoref{prop_rokhlin_system}.\\

However, we now arrive at an apparent point of departure between the two theories. In \cite[Theorem 3.9]{gardella_compact}, Gardella has shown that for actions of finite dimensional, compact groups, the restriction of an action with finite Rokhlin dimension to a subgroup also has finite Rokhlin dimension. In trying to extend this result to discrete groups, we arrived at an impasse. We were unable to prove the theorem in full generality, but we were able to prove it for a large class of groups (including all finitely generated, virtually abelian groups). It is to these ideas that we now turn and we begin with a lemma that shows that subgroups of finite index play a crucial role in this context.

\begin{lemma}\label{lem_finite_index_product}
Let $G$ be a residually finite group and $H\lf G$. Then there is a finite set $Y$ and an $H$-equivariant homeomorphism
\[
\theta : Y\times \overline{H} \to \overline{G}
\]
where $H$ acts on $Y$ trivially and by left-translation on both $\overline{H}$ and $\overline{G}$.
\end{lemma}
\begin{proof}
Let $\J_G := \{K \nf G : K\subset H\}$, then we claim that $\J_G$ is cofinal in $\mathcal{I}_G$. To see this, choose $L \in \mathcal{I}_G$. Since $H$ has finite index in $G$, there is a subgroup $L'\subset H$ such that $L'\nf G$. Then $K := L\cap L'\in \J_G$ and $K\subset L$. Therefore $\J_G$ is cofinal in $\I_G$ and we conclude from \cite[Lemma 1.1.9]{ribes} that
\[
\overline{G} \cong \varprojlim_{\J_G} G/K \text{ and } \overline{H} \cong \varprojlim_{\J_G} H/K.
\]
We need one more important ingredient: By a theorem of Ore \cite[Theorem 4.3]{ore}, there is a finite set $Y$ of left-coset representatives of $H$ in $G$ such that $Y$ is also a set of right-coset representatives. In other words,
\[
G = \bigsqcup_{y\in Y} yH = \bigsqcup_{y\in Y} Hy.
\]
We fix one such set $Y$ and define $\theta : Y\times \overline{H} \to \overline{G}$ by
\[
\theta(w, (h_KK)_{K\in \J_G}) := (h_KwK)_{K\in \J_G}
\]
To see that $\theta$ is well-defined, let $(w, (h_KK))\in Y\times \overline{H}$ and fix $K, L\in \J_G$ with $K\subset L$. Then $h_KwL = h_KLw = h_LLw = h_LwL$. Therefore $\theta(w, (h_KK)) \in \G$. Also, if $(w, (h_KK)) = (v,(g_KK)) \in Y\times \overline{H}$ are equal, then $w=v$ and $h_KK = g_KK$ for all $K \in \J_G$. Hence,
\[
h_KwK = h_KKw = g_KKw = g_KwK = g_KvK.
\]
Therefore, $\theta(w, (h_KK)) = \theta(v, (g_KK))$. We claim that $\theta$ is the homeomorphism we are looking for.
\begin{enumerate}
\item $\theta$ is continuous: Let $(w, (h_KK)) \in Y\times \overline{H}$ and let $\overline{g} = (g_KK) = \theta(w, (h_KK)) \in \G$. Let $U$ be an open set in $\G$ containing $\overline{g}$, then there is a finite set $F\subset \J_G$ such that
\[
V = \bigcap_{L\in F} (\pi^G_L)^{-1}(\{g_LL\}) = \bigcap_{L\in F}(\pi^G_L)^{-1}(\{h_LwL\}) \subset U.
\]
(Note that we write $\pi^G_L : \G\to G/L$ and $\pi^H_L : \overline{H}\to H/L$ for the natural maps.) Let $(v,(t_KK))\in \theta^{-1}(V)$ and fix $L \in F$. Then, $t_LvL = h_LwL$. Since $L\< G$, this implies that $t_LLv = h_LLw$. Since $L\subset H$ and $t_L, h_L \in H$, it follows that $Hv = Hw$. Since $Y$ is a set of right coset representatives of $H$, we conclude that $v=w$. Hence,
\[
t_LwL = h_LwL.
\]
Since $L\< G$, $h_LL  =t_LL$. This is true for each $L \in F$, so $(t_KK) \in W$ where
\[
W = \bigcap_{L\in F} (\pi^H_L)^{-1}(\{h_LL\}).
\]
Therefore, $\{w\} \times W$ is an open set in $Y\times \overline{H}$, it is contained in $\theta^{-1}(V) \subset \theta^{-1}(U)$ and $\overline{g} \in \{w\}\times W$. This is true for any $(w,(h_KK))\in Y\times \overline{H}$, so $\theta$ is continuous.
\item $\theta$ is injective: If $\theta(w, (h_KK)) = \theta(v,(t_KK))$, then for any $L \in \J_G$, $h_LwL = t_LvL$. As above, this implies that $Hw = Hv$. Since $Y$ is a set of right coset reprentatives, it follows that $w=v$. Then,
\[
h_LwL = t_LwL
\]
holds for all $L \in \J_G$. Since each $L\in \J_G$ is normal in $G$, we conclude that $h_LL = t_LL$, so that $(h_KK) = (t_KK)$. 
\item $\theta$ is surjective: If $\overline{g} = (g_LL) \in \overline{G}$, then for each $L\in \J_G$, there exists $w_L \in Y$ and $h_L \in H$ such that $g_L = w_Lh_L$. Since $\overline{g} \in \G$, it follows that whenever $K,L \in \J_G$ with $K\subset L$, we have
\[
w_Kh_KL = w_Lh_LL.
\]
Once again, this implies that $w_KH = w_LH$. Since $Y$ is a set of left coset representatives of $H$ in $G$, we conclude that $w_K = w_L$ whenever $K\subset L$. Now suppose $L_1, L_2\in \J_G$, then there exists $K\in \J_G$ such that $K\subset L_1\cap L_2$. Then it follows that
\[
w_{L_1} = w_K = w_{L_2}.
\]
Therefore, there is a commmon value $w := w_L$ for all $L\in \J_G$. Now consider $\overline{h} = (h_LL) \in \prod_{L\in \J_G} H/L$. The fact that $\overline{g}\in \G$ implies that whenever $K,L\in \J_G$ with $K\subset L$, one has
\[
wh_KL = wh_LL.
\]
Once again, since $L\< G$, this implies that $h_KL = h_LL$. We conclude that $\overline{h} \in \overline{H}$ and thus $\overline{g} = \theta(w, \overline{h})$. 
\end{enumerate}
Since $Y\times \overline{H}$ is compact and $\overline{G}$ is Hausdorff, $\theta$ is a homeomorphism. That $\theta$ is $H$-equivariant is easily verified. This completes the proof.

\end{proof}

\begin{definition}
Let $G$ be a discrete group and $H$ be a subgroup of $G$. We say that $H$ is a retract of $G$ if there is a group homomorphism $\rho : G\to H$ which restricts to the identity map on $H$. We say that $H$ is a virtual retract of $G$ (denoted by $H\leq_{vr} G$) if there exists $K\lf G$ such that $H$ is a retract of $K$.
\end{definition}

Suppose $H$ is a retract of a group $K$ and $\rho : K\to H$ is a group homomorphism such that $\rho\lvert_H = \mathrm{id}_H$. For each $L \<_{fin} H$, $\rho^{-1}(L) \<_{fin} K$ and there is a natural isomorphism $\rho_L : K/\rho^{-1}(L) \to H/L$. Since $\rho\lvert_H = \text{id}_H$, this map respects the left-translation action of $H$ on either group. Moreover, the universal property of the profinite completion (part (5) of \autoref{prop_profinite_completion}) allows us to build a continuous group homomorphism $\overline{\rho} : \overline{K} \to \overline{H}$ such that $\overline{\rho}\lvert_K = \rho$. In particular, $\overline{\rho}$ also respects the action of $H$ on both $\overline{K}$ and $\overline{H}$. It is this property that will be crucial to us in the following proof.

\begin{theorem}\label{thm_restriction_subgroup}
Let $G$ be a residually finite group and $H$ be a subgroup of $G$. Let $\alpha : G\to \Aut(A)$ be an action of $G$ on a C*-algebra $A$ and $\alpha_H : H\to \Aut(A)$ be the restricted action of $H$ on $A$. If $H\leq_{vr} G$, then $\dr(\alpha_H) \leq \dr(\alpha)$ and $\dr^c(\alpha_H) \leq \dr^c(\alpha)$.
\end{theorem}
\begin{proof}
We prove the first inequality since the argument for the second is similar. Assume without loss of generality that $d := \dr(\alpha) < \infty$. Choose $K\lf G$ and a homomorphism $\rho : K\to H$ that is identity on $H$. By \autoref{lem_finite_index_product}, there is a finite set $Y$ and a $K$-equivariant homeomorphism $\theta : \G\to Y\times \overline{K}$. Projecting onto the second component gives us a continuous, surjective, $K$-equivariant map $\mu : \G\to \overline{K}$. Moreover, as mentioned above, there is a continuous group homomorphism $\overline{\rho} : \overline{K} \to \overline{H}$ that respects the left-translation action of $H$ on both $\overline{K}$ and $\overline{H}$. Therefore, $p := \overline{\rho}\circ \mu : \G\to \overline{H}$ is a continuous $H$-equivariant map, which induces a unital $\ast$-homomorphism
\[
p^{\ast} : C(\overline{H}) \to C(\G)
\]
that is equivariant with respect to the action of $H$ on both algebras. To prove that $\dr(\alpha_H) \leq d$, choose finite sets $F\ssubset A, S\ssubset C(\overline{H}), M\ssubset H$ and $\epsilon > 0$. By hypothesis, there exist $(d+1)$ c.c.p. maps
\[
\psi_0, \psi_1, \ldots, \psi_d : C(\G) \to A
\]
which form a $(d,F, M, p^{\ast}(S),\epsilon)$-Rokhlin system (by \autoref{prop_rokhlin_system}). Hence, the system $\{\psi_{\ell}\circ p^{\ast} : 0\leq \ell\leq d\}$ forms a $(d,F, M, S,\epsilon)$-Rokhlin system for the action $\alpha_H$. We conclude that $\dr(\alpha_H) \leq d$.
\end{proof}

To understand how \autoref{thm_restriction_subgroup} may be used, we describe a few notions from geometric group theory. These have been explored in \cite{minasyan} and the reader will find a wealth of information concerning these ideas in that article.

\begin{definition}\label{defn_vrc_lr}
Let $G$ be a group. We say that $G$ has property (VRC) (for `virtual retractions onto cyclic subgroups') if every cyclic subgroup is a virtual retract of $G$. We say that $G$ has property (LR) (for `local retractions') if every finitely generated subgroup is a virtual retract of $G$.
\end{definition}

\begin{rem}\label{rem_vrc_lr}
We list a number of facts concerning these properties. Once again, the reader is referred to \cite{minasyan} for proofs of all of these.
~\begin{enumerate}
\item If $G$ satisfies (VRC), then it is residually finite.
\item If $G$ is a residually finite group and $H < G$ is a finite subgroup, then $H\leq_{vr} G$.
\item If $K, H$ are subgroups of $G$ such that $K \leq_{vr} H$ and $H\leq_{vr} G$, then $K\leq_{vr} G$.
\item Clearly, if $G$ satisfies (LR), then it satisfies (VRC), but the converse is false. However, if $G$ satisfies (VRC), then every finitely generated virtually abelian subgroup is a virtual retract of $G$.
\item If $G$ is a finitely generated, virtually abelian group, then every subgroup $H$ of $G$ is a virtual retract of $G$.
\item Free groups satisfy (LR) and virtually free groups satisfy (VRC).
\item If $K$ is a finite index subgroup of $G$ and $K$ satisfies (VRC), then $G$ satisfies (VRC).
\item The Heisenberg group $H$ (the group of all $3\times 3$ unitriangular matrices with integer coefficients) does not satisfy (VRC).
\end{enumerate}
\end{rem}

The following consequence of part (5) of \autoref{rem_vrc_lr} and \autoref{thm_restriction_subgroup} bears repeating.

\begin{cor}
Let $G$ be a finitely generated, virtually abelian group and $H$ be a subgroup of $G$. Let $\alpha : G\to \Aut(A)$ be an action of $G$ on a C*-algebra $A$ and $\alpha_H : H\to \Aut(A)$ be the restricted action of $H$ on $A$. Then $\dr(\alpha_H) \leq \dr(\alpha)$ and $\dr^c(\alpha_H) \leq \dr^c(\alpha)$.
\end{cor}
\section{Actions on $C_0(X)$-algebras and Commutative C*-algebras}\label{sec_cx_algebras}

Given a group $G$ and a locally compact metric space $X$, every action $\widetilde{\alpha} : G\curvearrowright X$ of $G$ on $X$ induces an action $\alpha : G\to \Aut(C_0(X))$ and vice-versa. If $G$ is compact and finite dimensional, then $\widetilde{\alpha}$ is free if and only if $\alpha$ has finite Rokhlin dimension (see \cite[Theorem 4.1]{gardella_compact} and \cite[Lemma 4.1]{self_rokhlin}). However, if $G$ is a discrete group, then the analogous statement is not known in full generality. What we do know is the following result, due to Szab\'{o}, Wu and Zacharias. Note that the version stated here (for locally compact spaces) is more general than the one proved in \cite{szabo_wu_zacharias}. However, Szab\'{o} has proved this more general statement in his PhD thesis (see \cite[Remark 8.7]{szabo_wu_zacharias}).

\begin{theorem}\cite[Corollary 8.5]{szabo_wu_zacharias}
Let $G$ be an infinite, finitely generated, nilpotent group and $X$ be a locally compact metric space of finite covering dimension. If $\widetilde{\alpha} : G\curvearrowright X$ is a free action of $G$ on $X$, then the induced action $\alpha : G\to \Aut(C_0(X))$ has finite Rokhlin dimension.
\end{theorem}

Our goal in this section is to investigate this question further, and the first observation is the following.

\begin{prop}\label{prop_frd_implies_free}
Let $X$ be a locally compact Hausdorff space and $G$ be a residually finite group. Let $\widetilde{\alpha} : G\curvearrowright X$ be such that the induced action $\alpha : G\to \Aut(C_0(X))$ has finite Rokhlin dimension. Then $\widetilde{\alpha}$ is free.
\end{prop}
\begin{proof}
Let $d := \dr(\alpha) < \infty$ and suppose $x\in X$ and $g\in G$ are such that $\widetilde{\alpha_g}^{-1}(x) = x$. If $g\neq e$, then we may choose $H\nf G$ such that $g\notin H$. Choose a function $f\in C_0(X)$ such that $f(x) = 1$. Let $k := [G:H], F := \{f\} \ssubset C_0(X), M := \{g\} \ssubset G, S := \{\delta_{\overline{s}} : \overline{s} \in G/H\} \ssubset C(G/H)$ and fix an $\epsilon > 0$. By \autoref{lem_rokhlin_system_finite_index}, choose $(d+1)$ c.c.p. maps $\psi_0, \psi_1, \ldots, \psi_d : C(G/H) \to C_0(X)$ which form a $(H,d,F,M,S,\epsilon)$-Rokhlin system. For $\overline{s} \in G/H$ and $0\leq \ell\leq d$, define $y_{\overline{s}}^{(\ell)} := \psi_{\ell}(\delta_{\overline{s}})$. Then,
\begin{eqsplit}
y_{\overline{s}}^{(\ell)}(x)^2 &= y_{\overline{s}}^{(\ell)}(x)y_{\overline{s}}^{(\ell)}(\widetilde{\alpha_g}^{-1}(x)) \\
&= y_{\overline{s}}^{(\ell)}(x) \alpha_g(y_{\overline{s}}^{(\ell)})(x) \\
&\approx_{\epsilon} y_{\overline{s}}^{(\ell)}(x)y_{\overline{gs}}^{(\ell)}(x) \\
&\approx_{\epsilon} 0.
\end{eqsplit}
Hence $y_{\overline{s}}^{(\ell)}(x) \approx_{\sqrt{2\epsilon}} 0$ for each $\overline{s} \in G/H$ and $0\leq \ell \leq d$. Therefore,
\[
1 \approx_{\epsilon} \sum_{\ell=0}^d \sum_{\overline{s} \in G/H} y_{\overline{s}}^{(\ell)}(x) \approx_{k(d+1)\sqrt{2\epsilon}} 0.
\]
This must hold for each $\epsilon > 0$ which is absurd. Therefore, $g = e$ must hold and we conclude that the action is free.
\end{proof}

To understand the extent to which the converse of \autoref{prop_frd_implies_free} holds for arbitrary residually finite groups (as against just nilpotent ones), we first study actions on $C_0(X)$-algebras. Our goal is to prove an estimate for the Rokhlin dimension of certain actions on a $C_0(X)$-algebra analogous to that of \cite[Theorem 2.3]{self_rokhlin}.

\begin{definition}
Let $X$ be a locally compact Hausdorff space. A C*-algebra $A$ is said to be a $C_0(X)$-algebra if there is a non-degenerate $\ast$-homomorphism $\Theta : C_0(X) \to Z(M(A))$, where $Z(M(A))$ denotes the center of the multiplier algebra of $A$.
\end{definition}

For a function $f\in C_0(X)$ and $a\in A$, we will write $fa := \Theta(f)(a)$. If $Y\subset X$ is a closed subspace, let $C_0(X,Y)$ denote the ideal of functions that vanish on $Y$. Then $C_0(X,Y)A$ is a closed ideal in $A$. We write $A(Y) := A/C_0(X,Y)A$ for the corresponding quotient and $\pi_Y : A\to A(Y)$ for the quotient map. If $Y = \{x\}$ is a singleton set, then the algebra $A(x) := A(\{x\})$ is called the fiber of $A$ at $x$, and we write $\pi_x : A\to A(x)$ for the corresponding quotient map. If $a\in A$, we simply write $a(x) := \pi_x(a) \in A(x)$. For each $a\in A$, we have a map $\Gamma_a : X\to \R$ given by $x \mapsto \|a(x)\|$. This map is upper semicontinuous (by \cite[Proposition 1.2]{rieffel}), a fact we will use crucially in the next proof. \\

Given a $C_0(X)$-algebra $A$, an automorphism $\beta \in \Aut(A)$ is said to be $C_0(X)$-linear if $\beta(fa) = f\beta(a)$ for all $f\in C_0(X)$ and all $a\in A$. We write $\Aut_X(A)$ for the subgroup of $\Aut(A)$ consisting of all $C_0(X)$-linear automorphisms. Given an automorphism $\beta \in \Aut_X(A)$ and a closed subset $Y \subset X$, there is a natural automorphism $\beta_Y \in \Aut(A(Y))$ such that $\beta_Y \circ \pi_Y = \pi_Y \circ \beta$. In particular, if $\alpha : G\to \Aut_X(A)$ is an action of a discrete group $G$ on $A$ by $C_0(X)$-linear automorphisms, then for each closed subspace $Y\subset X$, there is an action $\alpha_Y : G\to \Aut(A(Y))$ such that the quotient map $\pi_Y : A\to A(Y)$ is $G$-equivariant. Once again, if $Y = \{x\}$ is a singleton set, we write $\alpha_x$ for the action on $A(x)$. \\

We begin with a variant of \cite[Theorem 4.26]{gardella_lupini} that we need below.

\begin{lemma}\label{lem_almost_equivariant_section}
Let $A$ be a separable, nuclear C*-algebra and $G$ be an amenable group. Let $\alpha : G\to \Aut(A)$ and $\beta : G\to \Aut(B)$ be two actions, and $\pi : A\to B$ be a surjective $G$-equivariant $\ast$-homomorphism. Then for each $F\ssubset B, M\ssubset G$ and each $\epsilon > 0$, there is a c.c.p. section $\theta:B\to A$ such that
\[
\|\alpha_t(\theta(b)) - \theta(\beta_t(b))\| < \epsilon
\]
for all $b\in F$ and $t \in M$.
\end{lemma}
\begin{proof}
That there is a c.c.p. section $\widetilde{\theta} : B\to A$ follows from the Choi-Effros theorem. Let $(F_n)$ be a F\o lner sequence in $G$, and define
\[
\theta_n(b) := \frac{1}{|F_n|}\sum_{s\in F_n} \alpha_s(\widetilde{\theta}(\beta_{s^{-1}}(b)).
\]
Then each $\theta_n$ is c.c.p. (because it is a convex combination of c.c.p. maps). Since $\widetilde{\theta}$ is a section and $\pi\circ \alpha_s = \beta_s\circ \pi$, it follows that $\theta_n$ is a section. Now observe that if $t\in G$ and $b\in F$, then
\begin{equation*}
\begin{split}
\alpha_t(\theta_n(b)) - \theta_n(\beta_t(b)) &= \frac{1}{|F_n|} \left( \sum_{s\in F_n} \alpha_{ts}(\widetilde{\theta}(\beta_{s^{-1}}(b))) - \sum_{s\in F_n} \alpha_s(\widetilde{\theta}(\beta_{s^{-1}t}(b)))\right) \\
&= \frac{1}{|F_n|}\left(\sum_{w\in tF_n} \alpha_{w}(\widetilde{\theta}(\beta_{w^{-1}t}(b))) - \sum_{w\in F_n} \alpha_{w}(\widetilde{\theta}(\beta_{w^{-1}t}(b)))\right)
\end{split}
\end{equation*}
Hence,
\[
\|\alpha_t(\theta_n(b)) - \theta_n(\beta_t(b))\| \leq \frac{|tF_n\triangle F_n|\|b\|}{|F_n|}.
\]
Therefore, $\theta = \theta_n$ does the job for $n$ large enough.
\end{proof}

\begin{theorem}\label{thm_cx_algebra}
Let $X$ be a locally compact Hausdorff space and $A$ be a separable, nuclear $C_0(X)$-algebra. Let $G$ be an amenable, finitely generated, residually finite group and $\alpha : G\to \Aut_X(A)$ be an action of $G$ on $A$ by $C_0(X)$-linear automorphisms. Then,
\[
\dr(\alpha) \leq (\dim(X) + 1)(\sup_{x\in X} \dr(\alpha_x) + 1) - 1.
\]
\end{theorem}
\begin{proof}
Assume without loss of generality that $n := \dim(X) < \infty$ and that $d := \sup_{x\in X} \dr(\alpha_x) < \infty$, and let  $m := (n+1)(d+1)-1$. Fix $F \ssubset A, S\ssubset C(\G), M\ssubset G$ and $\epsilon > 0$ and assume that $\|a\|\leq 1$ for all $a\in F$ and that $\|f\| \leq 1$ for all $f\in S$. Set $\eta := \frac{\epsilon}{2(m+2)}$ and choose a compact set $K\subset X$ such that
\[
\sup_{x\in X\setminus K}\|a(x)\| < \eta
\]
for all $a\in F$. Now observe that $A(K)$ is a $C(K)$-algebra, which carries an action $\alpha_K : G\to \Aut_K(A(K))$ of $G$ acting by $C(K)$-linear automorphisms. Moreover, $\dim(K) \leq \dim(X)$ and if $x\in K$, then $A(K)(x) = A(x)$. Now assume for a moment that there exist $(m+1)$ c.c.p. maps
\[
\psi_0, \psi_1, \ldots, \psi_m : C(\G)\to A(K)
\]
which form a $(m, \pi_K(F), M, S, \eta)$-Rokhlin system for the action $\alpha_K$. Then, we may choose any c.c.p. section $\theta : A(K) \to A$ (which exists by the Choi-Effros theorem), and observe that the maps $\{\theta\circ \psi_{\ell} : 0\leq \ell\leq m\}$ form a $(m, F, M, S, \epsilon)$-Rokhlin system for the action $\alpha$. This calculation is elementary and relies on the fact that for any $b\in A$ and $a\in F$,
\[
\|ba\| \leq \max\{\eta\|b\|, \|\pi_K(ba)\|\} \text{ and } \|[b,a]\| \leq \max\{2\eta\|b\|, \|[\pi_K(b), \pi_K(a)]\|\}.
\]
In turn, both of these follow from the fact that for any $c\in A, \|c\| = \sup_{x\in X}\|c(x)\|$ by \cite[Proposition 2.8]{blanchard}. Replacing $X$ by $K$, it therefore suffices to prove the theorem when $X$ is itself compact. \\

We now assume $X$ is compact. The remainder of the proof follows that of \cite[Theorem 2.3]{self_rokhlin}.  For $x\in X$ fixed, since $d \geq \dr(\alpha_x)$, there are $(d+1)$ c.c.p. maps
\[
\psi_0, \psi_1, \ldots, \psi_d : C(\G)\to A(x)
\]
which form a $(d, \pi_x(F), M, S, \eta)$-Rokhlin system. By the Choi-Effros theorem (which is applicable since $G$ is finitely generated), there are c.c.p. maps $\widetilde{\psi}_0, \widetilde{\psi}_1, \ldots, \widetilde{\psi}_d : C(\G) \to A$ such that $\pi_x \circ \widetilde{\psi}_j = \psi_j$ for all $0\leq j\leq d$. Since the maps $\{\Gamma_c : c\in A\}$ are upper semicontinuous, for each $x\in X$, there is an open set $V_x$ containing $x$ and $(d+1)$ c.c.p. maps $\psi^{(0)}, \psi^{(1)}, \ldots, \psi^{(d)} : C(\overline{G})\to A(\overline{V_x})$ satisfying the following conditions:
\begin{enumerate}
\item $\psi^{(k)}(f_1)\psi^{(k)}(f_2)a\approx_{\eta} 0$ for all $f_1, f_2\in S$ with $f_1\perp f_2$, all $a\in F$ and all $0\leq k\leq d$.
\item $\sum_{k=0}^d \psi^{(k)}(1_{C(\G)})\pi_{\overline{V_x}}(a) \approx_{\eta} \pi_{\overline{V_x}}(a)$ for all $a\in F$.
\item $[\psi^{(k)}(f), \pi_{\overline{V_x}}(a)]\approx_{\eta} 0$ for all $a\in F, f\in S$ and $0\leq k\leq d$.
\item $(\alpha_s)_{\overline{V_x}}(\psi^{(k)}(f)) \approx_{\eta} \psi^{(k)}(\sigma_s(f))$ for all $f\in S, s\in M$ and $0\leq k\leq d$.
\end{enumerate}
By \cite[Lemma 2.2]{self_rokhlin}, we may choose a strongly $n$-decomposable refinement of this cover $\mathcal{U} = \mathcal{U}_0\sqcup \mathcal{U}_1 \sqcup \ldots \sqcup \mathcal{U}_n$. In other words, if $\mathcal{U}_i = \{V_{i,1}, V_{i,2}, \ldots, V_{i,k_i}\}$, then $\overline{V_{i,j_1}} \cap \overline{V_{i,j_2}} = \emptyset$ whenever $j_1\neq j_2$. Now for each $0\leq i\leq n$ and $1\leq j\leq k_i$, there are $(d+1)$ c.c.p. maps
\[
\psi^{(0)}_{i,j}, \psi^{(1)}_{i,j}, \ldots, \psi^{(d)}_{i,j} : C(\G)\to A(\overline{V_{i,j}}).
\]
satisfying the four conditions listed above on $\overline{V_{i,j}}$. Write $V_i = \sqcup_{j=1}^{k_i} V_{i,j}$ so that $\overline{V_i} = \sqcup_{j=1}^{k_i} \overline{V_{i,j}}$. By \cite[Lemma 2.4]{mdd_finite},
\[
A(\overline{V_i}) \cong \bigoplus_{j=1}^{k_i} A(\overline{V_{i,j}}).
\]
Therefore, we get $(d+1)$ c.c.p. maps
\[
\psi_i^{(0)}, \psi_i^{(1)}, \ldots, \psi_i^{(d)} : C(\G)\to A(\overline{V_i})
\]
satisfying the four conditions listed above on $\overline{V_i}$. Choose a partition of unity $\{f_i : 0\leq i\leq n\}$ subordinate to the open cover $\{V_0, V_1, \ldots, V_n\}$ and unital, c.c.p. sections $\theta_i : A(\overline{V_i}) \to A$ using \autoref{lem_almost_equivariant_section} such that
\[
\alpha_t(\theta_i(\psi_i^{(k)}(f)) \approx_{\eta} \theta_i((\alpha_t)_{\overline{V_i}}(\psi_i^{(k)}(f)))
\]
for all $f\in S$ and $t\in M$. For $0\leq i\leq n, 0\leq k\leq d$, define $\varphi_{i,k} : C(\overline{G})\to A$ by
\[
\varphi_{i,k}(f) := f_i\theta_i(\psi_i^{(k)}(f)).
\]
Then each $\varphi_{i,k}$ is an c.c.p. map, and an argument entirely similar to \cite[Theorem 2.3]{self_rokhlin} (with $\G$ playing the role of $G$) shows that the collection $\{\varphi_{i,k} : 0\leq i\leq n, 0\leq k\leq d\}$ forms a $(m, F,M, S, \epsilon)$-Rokhlin system for $\alpha$. Thus, $\dr(\alpha) \leq (n+1)(d+1) - 1$.
\end{proof}

The next result now wraps up the discussion (to an extent) of the relationship between free actions on a locally compact space $X$ and finiteness of Rokhlin dimension of the associated action on $C_0(X)$.

\begin{cor}\label{thm_free_proper}
Let $G$ be an amenable, finitely generated, residually finite group and let $X$ be a locally compact, separable metric space of finite covering dimension. Let $\tilde{\alpha}:G\curvearrowright X$ be an action of $G$ on $X$ and let $\alpha : G\to \Aut(C_0(X))$ be the induced action on $C_0(X)$. If $\tilde{\alpha}$ is free and proper, then
\[
\dr(\alpha) \leq \dim(X).
\]
\end{cor}
\begin{proof}
Let $Y := X/G$, which is Hausdorff and locally compact and the natural map $\pi : X\to Y$ is a local homeomorphism. Moreover, $X$ is metrizable and separable, so by Alexandroff's theorem \cite[Theorem 1.12.8]{engelking} it follows that $\dim(Y) = \dim(X) < \infty$ (Note that small inductive dimension and Lebesgue covering dimension coincide for separable metric spaces). \\

If $A := C_0(X)$, the map $\pi^{\ast} : C_0(Y) \to Z(M(A)) \cong C_b(X)$ gives $A$ the structure of a $C_0(Y)$-algebra. $A$ is separable because $X$ is second countable and $A$ is clearly nuclear. Moreover, the action $\alpha$ is by $C_0(Y)$-linear automorphisms. By \cite[Example C.4]{williams}, the fiber $A(y)$ at a point $y = \pi(x)\in Y$ is isomorphic to $C_0(G\cdot x)$. The map $\theta : G\to G\cdot x$ given by $g \mapsto \alpha_g(x)$ induces an equivariant isomorphism
\[
\theta^{\ast} : (C_0(G\cdot x), \alpha_y) \to (C_0(G), \text{Lt})
\]
where $\text{Lt} : G\to \Aut(C_0(G))$ is the natural action of $G$ on $C_0(G)$ induced by the left-translation action of $G$ on itself. We claim that $\dr(\text{Lt}) = 0$. To see this, fix $H\nf G$. Then the quotient map $p : G\to G/H$ induces an equivariant, unital $\ast$-homomorphism
\[
p^{\ast} : (C(G/H), \sigma_G) \to (C_b(G), \text{Lt}).
\]
By \autoref{lem_rokhlin_system_finite_index}, the action $\text{Lt} : G\to \Aut(C_b(G))$ has Rokhlin dimension zero. Since $C_0(G)$ is a $G$-invariant hereditary subalgebra of $C_b(G)$, it follows from \autoref{thm_permanence} that $\text{Lt} : G\to \Aut(C_0(G))$ also has Rokhlin dimension zero. Hence $\dr(\alpha_y) = 0$ for each $y\in Y$. The result now follows from \autoref{thm_cx_algebra}.
\end{proof}

\section{Ideal Separation and Outerness}\label{sec_outerness}

In this, the final section of the paper, we discuss three related notions for an action of a residually finite group on a C*-algebra with finite Rokhlin dimension. We show that such actions are pointwise outer. Using a theorem of Sierakowski \cite{sierakowski}, we show that the ideals in the associated reduced crossed product C*-algebra must arise from invariant ideals in the underlying C*-algebra. Finally, we show how this property implies that such actions are also properly outer, provided the group satisfies the (VRC) property.

\begin{definition}
Let $A$ be a C*-algebra. An automorphism $\alpha \in \Aut(A)$ is said to be inner if there is a unitary $u \in \widetilde{A}$ such that $\alpha(a) = uau^{\ast}$ for all $a\in A$. In that case, we write $\alpha = \text{Ad}(u)$. Moreover, $\alpha$ is said to be outer if it is not inner. If $G$ is a locally compact group, an action $\alpha : G\to \Aut(A)$ is said to be pointwise outer if $\alpha_g$ is outer for each non-identity element $g\in G$.
\end{definition}

If $\alpha : G\to \Aut(A)$ is an action of a compact group $G$ on a C*-algebra $A$, then finiteness of Rokhlin dimension implies that $\alpha$ is pointwise outer (\cite[Proposition 4.15]{gardella_compact}). We show that the same is true for discrete, residually finite groups as well.

\begin{prop}
Let $A$ be a C*-algebra and $G$ be a residually finite group. If $\alpha: G\to \Aut(A)$ is an action such that $\dr(\alpha) < \infty$, then $\alpha$ is pointwise outer.
\end{prop}
\begin{proof}
Let $d := \dr(\alpha) < \infty$ and suppose $g\in G$ is a non-identity element such that $\alpha_g = \text{Ad}(u)$ for some unitary $u\in \widetilde{A}$. Write $u = v + \lambda 1_{\widetilde{A}}$ for some $v\in A$ and $\lambda \in \C$. Choose a subgroup $H\lf G$ such that $g\notin H$ and let $n := [G:H]$. Fix a non-zero element $x\in A$ with $\|x\| \leq 1$, and set $F := \{x, x^{\ast}, x^{1/2}, v, v^{\ast}\} \ssubset A, S := \{\delta_{\overline{s}} : \overline{s} \in G/H\} \ssubset C(G/H), M = \{g\}$ and choose $\epsilon > 0$. By \autoref{lem_rokhlin_system_finite_index}, there exist $(d+1)$ c.c.p. maps
\[
\varphi_0, \varphi_1, \ldots, \varphi_d : C(G/H) \to A
\]
which form an $(H, d, F, M, S, \epsilon)$-Rokhlin system. For $0\leq \ell\leq d$ and $\overline{s} \in G/H$, let $y_{\overline{s}}^{(\ell)} := \varphi_{\ell}(\delta_{\overline{s}})$. Then,
\[
x^{1/2}uy_{\overline{s}}^{(\ell)}u^{\ast}x^{1/2} = x^{1/2} \alpha_g(y_{\overline{s}}^{(\ell)})x^{1/2} \approx_{\epsilon} x^{1/2} y_{\overline{gs}}^{(\ell)}x^{1/2} \approx_{\epsilon} y_{\overline{gs}}^{(\ell)} x.
\]
However,
\begin{eqsplit}
x^{1/2}uy_{\overline{s}}^{(\ell)}u^{\ast}x^{1/2} &= x^{1/2}(vy_{\overline{s}}^{(\ell)}v^{\ast} + \overline{\lambda}vy_{\overline{s}}^{(\ell)} + \lambda y_{\overline{s}}^{(\ell)} v^{\ast} + |\lambda|^2 y_{\overline{s}}^{(\ell)})x^{1/2} \\
&\approx_{2\epsilon} x^{1/2}(y_{\overline{s}}^{(\ell)}vv^{\ast} + \overline{\lambda} y_{\overline{s}}^{(\ell)} v + \lambda y_{\overline{s}}^{(\ell)} v^{\ast} + |\lambda|^2 y_{\overline{s}}^{(\ell)}) x^{1/2} \\
&= x^{1/2} y_{\overline{s}}^{(\ell)}uu^{\ast} x^{1/2} \\
&\approx_{\epsilon} y_{\overline{s}}^{(\ell)} x
\end{eqsplit}
Since c.c.p. maps preserve adjoints, the $y_{\overline{s}}^{(\ell)}$ are self-adjoint, and
\[
(y_{\overline{s}}^{(\ell)}x)^{\ast} (y_{\overline{s}}^{(\ell)}x) \approx_{5\epsilon} (y_{\overline{gs}}^{(\ell)} x)^{\ast} (y_{\overline{s}}^{(\ell)}x) = x^{\ast} y_{\overline{gs}}^{(\ell)}y_{\overline{s}}^{(\ell)}x \approx_{\epsilon} 0.
\]
Hence, $y_{\overline{s}}^{(\ell)}x \approx_{\sqrt{6\epsilon}} 0$ for each $\overline{s} \in G/H$. This implies that
\[
x \approx_{\epsilon} \sum_{\ell=0}^d \sum_{\overline{s} \in G/H} y_{\overline{s}}^{(\ell)}x \approx_{n(d+1)\sqrt{6\epsilon}} 0.
\]
This is true for any $\epsilon > 0$, so $x=0$. This contradicts our assumption on $x$, so we conclude that $\alpha_g$ cannot be inner.
\end{proof}

We now turn our attention to the ideal structure of the reduced crossed product C*-algebra. Let $A$ be a C*-algebra and $\alpha : G\to \Aut(A)$ be an action of a discrete group $G$ on $A$. We write $A\rtimes_r G$ for the associated reduced crossed product C*-algebra. If $I$ is a $G$-invariant ideal of $A$, then $I\rtimes_r G$ forms an ideal in $A\rtimes_r G$. An interesting question, therefore, is to determine conditions under which every ideal of $A\rtimes_r G$ arises in this way.

\begin{definition}
For an action $\alpha : G\to \Aut(A)$, we say that $A$ separates ideals in $A\rtimes_r G$ if the only ideals in $A\rtimes_{r} G$ are of the form $I\rtimes_{r} G$ for some $G$-invariant ideal $I \< A$.
\end{definition}

The following result due to Sierakowski is relevant to us, as it gives an easily verifiable condition to determine if an action has the ideal separation property. Recall that if $\alpha : G\to \Aut(A)$ is an action of a discrete group $G$ on a C*-algebra $A$, then there is a conditional expectation $E : A\rtimes_r G\to A$ such that $E(\sum_{s\in G} a_s\lambda_s) = a_e$ on $C_c(G,A)$ (see, for instance, \cite[Proposition 4.1.9]{brown_ozawa}). Given a C*-algebra $B$ and an element $x\in B$, we write $I_B[x]$ for the ideal in $B$ generated by $x$.

\begin{theorem}\cite[Theorem 1.13]{sierakowski}\label{thm_sierakowski} Let $G$ be a discrete group and $\alpha : G\to \Aut(A)$ be an exact action of $G$ on a C*-algebra $A$. If $E(x) \in I_{A\rtimes_r G}[x]$ for every positive element $x\in A\rtimes_r G$, then $A$ separates ideals in $A\rtimes_r G$.
\end{theorem}

In \cite[Theorem 2.2]{pasnicu_phillips_2}, Pasnicu and Phillips have shown that if $G = \Z$ and the action has the Rokhlin property, then it has the ideal separation property. Sierakowski extended this result to include finite groups in \cite[Theorem 1.30]{sierakowski} (the result is originally due to Pasnicu and Phillips \cite[Corollary 2.5]{pasnicu_phillips}, although the proof does not rely on \autoref{thm_sierakowski}). The analogous result for actions of compact abelian groups with finite Rokhlin dimension was proved in \cite[Corollary 2.17]{gardella_hirshberg_santiago}. We now extend these result to actions of residually finite groups with finite Rokhlin dimension. The following lemma (whose proof we omit) will be useful to us.

\begin{lemma}\label{lem_orthogonal_contractions}
Let $B$ be a C*-algebra, $\{v_1, v_2, \ldots, v_n\}$ be self-adjoint, orthogonal contractions and $a,b\in B$. Then
\[
\left\|\sum_{i=1}^n v_iav_i - \sum_{i=1}^n v_ibv_i\right\| \leq \|a-b\|
\]
\end{lemma}

\begin{theorem}\label{thm_ideal_separation}
Let $A$ be a C*-algebra, $G$ be an exact, residually finite group and $\alpha : G\to \Aut(A)$ be an action of $G$ on $A$ with $\dr(\alpha) < \infty$. For any subgroup $H < G$, $A$ separates ideals in $A\rtimes_r H$.
\end{theorem}
\begin{proof}
Since $H$ is exact, the action $\alpha_H : H \to \Aut(A)$ is exact, so it suffices to verify the conditions of \autoref{thm_sierakowski}. Let $d := \dr(\alpha) < \infty$ and set $B := A\rtimes_{r} H$. Fix $x\in B^+$ and we wish to prove that $E(x) \in I_B[x]$. For $\epsilon > 0$ fixed, there exists $z \in C_c(H,A)$ such that $\|x-z\| < \frac{\epsilon}{3(d+1)}$. Let $M\ssubset H$ be a finite set such that $z = \sum_{t\in M} a_t u_t$. Assume $e\in M, \|a_t\| \leq 1$ for all $t\in M$ and choose $K \nf G$ such that $M\cap K = \{e\}$. For convenience of notation, we write $k := |M|$. Let $F = \{a_t : t\in M\} \ssubset A, S := \{\delta_{\overline{s}} : \overline{s} \in G/K\} \ssubset C(G/K)$ and fix $\eta > 0$ to be chosen later. Choose $0 < \delta < \eta$ such that if $d,d'\in A$ are two positive contractions, then
\begin{eqsplit}
\|[d,d']\| < \delta &\Rightarrow \|[\sqrt{d}, d']\| < \eta \text{ and } \\
\|d-d'\| < \delta &\Rightarrow \|\sqrt{d} - \sqrt{d'}\| < \eta.
\end{eqsplit}
The first condition can be met by \cite[Exercise 3.9.6]{higson_roe} and the second by \cite[Lemma 1.2.5]{rordam}. By \autoref{lem_rokhlin_system_finite_index},  there exist $(d+1)$ c.c.p. maps
\[
\varphi_0, \varphi_1, \ldots, \varphi_d : C(G/K) \to A
\]
which form a $(K, d, F, M, S, \delta)$-Rokhlin system. Moreover, since the cone over $C(G/K)$ is projective, we may arrange it so that
\[
\varphi_i(\delta_{\overline{s}})\varphi_i(\delta_{\overline{r}}) = 0
\]
whenever $\overline{s} \neq \overline{r}$ in $G/K$ (see \cite[Remark 1.18]{hirshberg_phillips} and \cite[Theorem 4.6]{loring}). For $\overline{s} \in G/K$ and $0\leq i\leq d$, define $y_{\overline{s}}^{(i)} := \sqrt{\varphi_i(\delta_{\overline{s}})}$. Then for all $a\in F$,
\begin{enumerate}
\item $\sum_{i=0}^d \sum_{\overline{s} \in G/K} (y^{(i)}_{\overline{s}})^2a \approx_{\delta} a$.
\item $y^{(i)}_{\overline{s}}y^{(i)}_{\overline{r}} = 0$ for any $\overline{s}, \overline{r} \in G/K$ with $\overline{s}\neq \overline{r}$. In particular, if $t\in M\setminus\{e\}$, then $y^{(i)}_{\overline{s}}y^{(i)}_{\overline{ts}} = 0$ for all $\overline{s} \in G/K$.
\item $[y^{(i)}_{\overline{s}},a] \approx_{\eta} 0$ for all $\overline{s} \in G/K$.
\item $\alpha_t(y^{(i)}_{\overline{s}})a \approx_{\eta} y^{(i)}_{\overline{ts}}a$ and $a\alpha_t(y^{(i)}_{\overline{s}}) \approx_{\eta} ay^{(i)}_{\overline{ts}}$ for any $t\in M$ and $\overline{s} \in G/K$.
\end{enumerate}
For any $\overline{s} \in G/K$ and $0\leq i\leq d$,
\begin{eqsplit}
y_{\overline{s}}^{(i)}z y_{\overline{s}}^{(i)} &= \sum_{t\in M} y_{\overline{s}}^{(i)} a_tu_t y_{\overline{s}}^{(i)} \\
&= \sum_{t\in M} y_{\overline{s}}^{(i)}a_t \alpha_t(y_{\overline{s}}^{(i)})u_t \\
&\approx_{k\eta} \sum_{t\in M} y_{\overline{s}}^{(i)} a_t y_{\overline{ts}}^{(i)}u_t \\
&\approx_{k\eta} \sum_{t\in M} y_{\overline{s}}^{(i)}y_{\overline{ts}}^{(i)} a_tu_t \\
&= (y_{\overline{s}}^{(i)})^2a_e.
\end{eqsplit}
Therefore,
\begin{eqsplit}
E(x) &\approx_{\frac{\epsilon}{3(d+1)}} E(z) \\
&= a_e \\
&\approx_{\eta} \sum_{i=0}^d \sum_{\overline{s}\in G/K} (y_{\overline{s}}^{(i)})^2 a_e \\
&\approx_{(d+1)(2k)\eta} \sum_{i=0}^d \sum_{\overline{s}\in G/K} y^{(i)}_{\overline{s}}zy^{(i)}_{\overline{s}} \\
&\approx_{(d+1)\frac{\epsilon}{3(d+1)}} \sum_{i=0}^d \sum_{\overline{s} \in G/K} y_{\overline{s}}^{(i)}xy_{\overline{s}}^{(i)}
\end{eqsplit}
where the last approximation follows from \autoref{lem_orthogonal_contractions}. If we choose $\eta > 0$ so that
\[
\eta < \frac{\epsilon}{3(1 + (d+1)(2k))},
\]
then,
\[
E(x) \approx_{\epsilon} \sum_{i=0}^d \sum_{\overline{s} \in G/K} y_{\overline{s}}^{(i)}xy_{\overline{s}}^{(i)} \in I_B[x].
\]
Hence, $E(x) \in I_B[x]$ for each $x\in B^+$, so by \cite[Theorem 1.13]{sierakowski}, $A$ separates ideals in $A\rtimes_r H$.
\end{proof}

We end with a short application of this result. If $A$ is a C*-algebra and $\alpha \in \Aut(A)$, we write $\text{Sp}(\alpha)$ for the Arveson spectrum of $\alpha$ (see \cite[Section 8.1]{pedersen}). If $H^{\alpha}(A)$ denotes the set of all non-zero, $\alpha$-invariant, hereditary C*-subalgebras of $A$, then the Connes spectrum of $\alpha$ is defined as
\[
\Gamma(\alpha) = \bigcap_{D \in H^{\alpha}(A)} \text{Sp}(\alpha\lvert_D)
\]
Write $H^{\alpha}_{\mathrm{B}}(A)$ for the set of all $D \in H^{\alpha}(A)$ with the property that the closed ideal generated by $D$ is an essential ideal of $A$. Then the Borchers spectrum of $\alpha$ is defined as
\[
\Gamma_{\mathrm{B}}(\alpha) = \bigcap_{D\in H^{\alpha}_{\mathrm{B}}(A)} \text{Sp}(\alpha\lvert_D)
\]
For more about these objects and their relationship with the ideal structure of the reduced crossed product, the reader is referred to \cite{pedersen}. The notion of proper outerness defined below is originally due to Kishimoto \cite{kishimoto_free}. For separable C*-algebras, this definition is equivalent to a number of other conditions (see \cite[Theorem 6.6]{olesen_pedersen_3}). For convenience, we adopt one such condition as our definition.

\begin{definition}
Let $A$ be a separable C*-algebra. An automorphism $\alpha \in \Aut(A)$ is said to be properly outer if $\Gamma_{\mathrm{B}}(\alpha\lvert_D) \neq \{1\}$ for each $D\in H^{\alpha}(A)$. Moreover, an action $\alpha : G\to \Aut(A)$ of a discrete group on $A$ is said to be properly outer if $\alpha_g$ is properly outer for each non-identity element $g\in G$.
\end{definition}

\begin{cor}
Let $A$ be a separable C*-algebra, $G$ be an amenable, residually finite group satisfying the (VRC) property, and let $\alpha : G\to \Aut(A)$ be an action of $G$ on $A$. If $\dr(\alpha) < \infty$, then $\alpha$ is properly outer.
\end{cor}
\begin{proof}
Fix $g\in G\setminus\{e\}$ and let $H$ be the cyclic subgroup generated by $g$. Since $G$ satisfies the (VRC) property, the restricted action $\alpha_H : H \to \Aut(A)$ also has finite Rokhlin dimension by \autoref{thm_restriction_subgroup}. Replacing $G$ by $H$, we may assume that $G$ is cyclic. We write $\beta$ for the automorphism $\alpha_g$ (and for the action $\beta : G\to \Aut(A)$), and write $\widehat{G}$ for the Pontryagin dual of $G$.

Now if $D$ is a non-zero, $\beta$-invariant, hereditary C*-subalgebra of $A$, then $\beta\lvert_D : G\to \Aut(D)$ also has finite Rokhlin dimension by \autoref{thm_permanence}. Note that $G$ is exact, so by \autoref{thm_ideal_separation}, each non-zero ideal in $D\rtimes_r G$ has non-zero intersection with $D$. By \cite[Theorem 2.5]{olesen_pedersen_2}, $\Gamma(\beta\lvert_D) = \widehat{G}$. Hence,
\[
\Gamma_{\mathrm{B}}(\beta\lvert_D) = \widehat{G}
\]
This is true for each $D \in H^{\alpha}(A)$, so $\beta$ is properly outer.
\end{proof}

\textbf{Acknowledgements:} The first named author is supported by NBHM Fellowship (Fellowship No. 0203/26/2022/R{\&}D-II/16171) and the second named author was partially supported by the SERB (Grant No. MTR/2020/000385). The authors would like to thank the anonymous referee for their comments which helped strengthen some of these results.

\end{document}